\documentclass[1p]{elsarticle}

\usepackage{lineno}%,hyperref}
\modulolinenumbers[5]

\journal{ }

\usepackage{mathrsfs}
\usepackage{amsfonts}
\usepackage{enumerate}
\usepackage{latexsym}
\usepackage[all,cmtip]{xy}
\biboptions{numbers,sort&compress}

\usepackage{verbatim}

\usepackage[colorlinks,linkcolor=blue,citecolor=blue]{hyperref}

\usepackage{amsmath,amssymb,xspace,amsthm}
\usepackage{bbm}

\numberwithin{equation}{section}

\def\fg{\mathfrak{g}}

\def\fo{\mathfrak{o}}

\def\fs{\mathfrak{s}}

\def\fW{\mathfrak{W}}

\def\cA{\mathcal{A}}

\def\cF{\mathcal{F}}

\def\cK{\mathcal{K}}
\def\cL{\mathcal{L}}
\def\cM{\mathcal{M}}

\def\cR{\mathcal{R}}
\def\cS{\mathcal{S}}
\def\cT{\mathcal{T}}

\def\cW{\mathcal{W}}

\def\sB{\mathscr{B}}

\def\sF{\mathscr{F}}

\def\sJ{\mathscr{J}}

\def\sW{\mathscr{W}}

\def\bC{\mathbb{C}}
\def\bM{\mathbb{M}}
\def\bN{\mathbb{N}}
\def\bZ{\mathbb{Z}}

\def\ra{\mathrm{a}}

\def\rd{\mathrm{d}}
\def\re{\mathrm{e}}

\def\rr{\mathrm{r}}

\def\rD{\mathrm{D}}

\def\Supp{\mathrm{Supp}}

\def\Vir{\mathbf{Vir}}

\def\zero{\mathbf{0}}
\def\one{\mathbf{1}}

\newtheorem{theorem}{{Theorem}}[section]
\newtheorem{lemma}[theorem]{Lemma}
\newtheorem{remark}[theorem]{Remark}

\newtheorem{proposition}[theorem]{Proposition}

\allowdisplaybreaks

\begin{document}

\begin{frontmatter}



\title{Classification of simple quasifinite modules for contact superconformal algebras with $N\ne4$}


%\author{Yan-an Cai, Rencai L\"{u}}

\author[1]{Yan-an Cai\fnref{fn1}}
\ead{yatsai@suda.edu.cn}

\author[1]{Rencai L\"{u}\corref{cor1}}
\ead{rlu@suda.edu.cn}

%\author[2]{Yan Wang\corref{cor1}\fnref{fn3}}
%\ead{wangyan09@tju.edu.cn}

\address[1]{Department of Mathematics, Soochow University, Suzhou 215006, Jiangsu, P. R. China}

%\address[2]{Department of Mathematics, Tianjin University, Tianjin 300072, P. R. China}

\cortext[cor1]{Corresponding author}

%\fntext[fn1]{Y. C. is partially supported by NSF of China (Grant 11801390)}

%\fntext[fn2]{R. L. is partially supported by NSF of China (Grant 11471233, 11771122, 11971440)}

%\fntext[fn3]{Y. W. is partially supported by NSF of China (Grant 11471233, 11771122, 11801390) and Postgraduate Research and Practice Innovation Program of Jiangsu Province (Grant KYCX17\_2014)}
%\end{comment}

\begin{abstract}
In this paper, we classify all simple jet modules for contact superconformal algebras $\cK(N;\epsilon)$ with $N\neq4$. Then all simple quasifinite modules for $\widehat{\cK}(N;\epsilon)$ ($N\neq4$), the universal central extension of $\cK(N;\epsilon)$, are classified. Our results show that Mat\'{i}nez-Zelmanov's conjecture in \cite{MZe1} holds for $\cK(N;\epsilon)$ ($N\ne4$).
\end{abstract}


\end{frontmatter}

%\linenumbers

\section{Introduction}\label{sec1}

We denote by $\bZ, \bZ_+, \bN, \bC$ and $\bC^*$ the sets of all integers, non-negative integers, positive integers, complex numbers, and nonzero complex numbers, respectively. All vector spaces and algebras in this paper are over $\bC$. We denote by $\delta_{i,j}$ the Kronecker delta. For a statement $P$, set
 \[
 \delta_P=\begin{cases}
 0, & \mbox{ if } P \mbox{ does not hold},\\
 1, & \mbox{ if } P \mbox{ holds}.
 \end{cases}
 \]
Throughout this paper, by subalgebras, submodules for Lie superalgebras we mean subsuperalgebras and subsupermodules respectively. For $N\in\bN$, set $\zero:=(0,\ldots,0), \one:=(1,\ldots,1)\in\bZ^N$ and $\overline{1,N}=\{1,2,\ldots,N\}$. For a real number $r$, denote by $\lfloor r\rfloor$ the maximal integer smaller than or equal to $r$. Let $\boldsymbol{i}=\sqrt{-1}$. $\Pi$ denotes the parity change functor.

In view of importance of the Witt algebra $\cW(0)=\rD\re\rr\bC[t^{\pm1}]$ and the Virasoro algebra (the central extension of $\cW(0)$) in physics, Neveu, Schwarz \cite{NS}, Ramond \cite{R} and others considered superextensions of the algebra $\cW(0)$. These superextensions became known as \emph{superconformal algebras}. A superconformal algebra (see \cite{KL}) is a simple complex Lie superalgebra such that it contains the Witt algebra as a subalgebra, and has growth $1$. The $\bZ$-graded superconformal algebras are ones for which $\ra\rd\, t\frac{d}{dt}$ is diagonalizable with finite dimensional eigenspaces, see \cite{KL}. %In general, a superconformal algebra is a subalgebra of the Lie superalgebra $\cW(N)$ of all derivations of $\bC[t^{\pm1}]\otimes\Lambda(N)$, where $\Lambda(N)$ is the Grassmann algebra in $N$ odd variables.
The main known series of superconformal algebras are $\cW(N) (N\geq0)$, all derivations of $\bC[t^{\pm1}]\otimes\Lambda(N)$, where $\Lambda(N)$ is the Grassmann algebra in $N$ odd variables, the series $\cS'(N;\alpha)(N\geq2)$ of one-parameter families of deformations of the divergence-free subalgebra of $\cW(N)$ and the series of \emph{contact superconformal algebras} $\cK'(N;\epsilon) (N\geq1)$, see \cite{Ka1, KL}. The corresponding central extensions were classified in \cite{KL}.

The simplest example is the Virasoro algebra. An important class of modules for the Virasoro algebra are the so-called quasifinite modules (or Harish-Chandra modules), the weight modules with finite dimensional weight spaces, which were classified by O. Mathieu in \cite{Ma}. A classification of simple quasifinite modules for $\cW(N)$, i.e. simple modules on which $t\frac{d}{dt}$ acts diagonally with finite dimensional weight spaces, was given in \cite{BFIK,XL1}.

In this paper, we will study simple quasifinite modules (also called Harish-Chandra modules, see \cite{MZe1}) for contact superconformal algebras, i.e. $\cK'(N;\epsilon) (\epsilon\in\bZ_+^N)$. Contact superconformal algebras have three sectors, Neveu-Schwarz sector ($\epsilon=\zero$), Ramond sector ($\epsilon=\one$) and the twisted sector ($\epsilon\in\{0,1\}^N\setminus\{\zero,\one\}$). In \cite{KL}, V. Kac and van de Leuer showed that Ramond sector is isomorphic to Neveu-Schwarz sector when $N$ is even, and is isomorphic to the twisted sector when $N$ is odd. Hence, there are essentially two different sectors.

The algebra $\cK(N;\zero)$ is also the Lie superalgebra of contact vector fields with Laurent polynomials as coefficients, which is characterized by its action on a contact $1$-form (see \cite{KL}); it is known to physicists as the $SO(N)$ superconformal algebra \cite{A,Sc}. $\cK(N;\epsilon)$ is simple except when $N=4$ and $\sum\limits_{i=1}^N\epsilon_i\in2\bZ$, in which case $\cK'(4;\epsilon)=[\cK(4;\epsilon),\cK(4;\epsilon)]$ is a simple ideal in $\cK(4;\epsilon)$ of codimension one.

Let $\widehat{\cK}(N;\epsilon)$ be the central extension of $\cK(N;\epsilon)$, which is nontrivial unless $N\le4$. The algebras $\widehat{\cK}(N;\epsilon)$ ($N\le4$) and $\widehat{\cS'}(2;\alpha)$ (central extension of $\cS'(2;\alpha)$) are also known as the super-Virasoro algebras.  Weight modules for the $N=1$ super-Viraoro algebras  have been extensively investigated (cf. \cite{DLM,IK1,IK2}), for more related results we refer the reader to \cite{CLX, CP, Rao, Ka, Ka1, KL, KS, LPX, Ma, MZe, MZ, S1} and references therein. A complete classification for $\widehat{\cK}(1;\epsilon)$ was given in \cite{S2}. Recently, with the theory of the $A$-cover introduced in \cite{BF},
a new approach to classify all simple quasifinite modules for $\widehat{K}(1;\epsilon)$ was given in \cite{CL, CLL}. The classification of simple quasifinite modules for $\widehat{\cK}(2;\one)$ was given in \cite{LPX}, \cite{XL1} and \cite{BFIK}. In this paper, we will classify all simple quasifinite modules for $\widehat{\cK}(N;\epsilon)$ with $N\neq4$.

One of the main step of the $A$-cover theory is to classify simple jet modules intorduced in \cite{B,Rao}. In \cite{XL1,CL,CLL}, the authors classify simple jet modules via computing the jet algebra. However, it would be much more complicated to compute jet algebras for $\widehat{K}(N;\epsilon)$ in general. Moreover, the methods to compute jet algebras for the Neveu-Schwarz algebra and the Ramond algebra are quite different. In this paper, with the weighting functor introduced in \cite{N}, we give a conceptional method to classify simple jet modules uniformly.

The paper is organized as follows. In Section \ref{pre}, we collect some basic results for our study. Simple jet modules are classified in Section \ref{cuspidalAK}. We classify all simple quasifinite modules for $\widehat{\cK}(N;\epsilon)$ ($N\neq4$) in Section \ref{Mainresults}. Simple quotients of ``tensor modules" for $N=2,3$ are given in Section \ref{N=2} and Section \ref{N=3}, respectively.

\

\begin{remark}
The main result of this paper (Theorem \ref{mainresult}) was reported by the first author in 18th National Conference  on Lie Theory on July 18th 2023, at Shanghai. During the visit of the second author to  Shenzhen International Center for Mathematics at Southern University of Science and Technology early March 2024, we learn from Prof. Zelmanov that there is a general result on classification of simple quasifinite modules for superconformal algebras of large rank in an unpublished manuscript joint with Mathieu and Martinez. We believe that together with methods in these two papers, we can get a classification of simple quasifinite modules for $K(4;\epsilon)$.
\end{remark}

\section{Preliminaries}
\label{pre}

Consider the associative algebra $\cA:=\bC[t^{\pm1}]\otimes\Lambda(N)$, where $\Lambda(N)$ is the Grassmann algebra in $N$ variables $\xi_1,\ldots,\xi_N$ and $|t|:=\bar{0}, |\xi_i|:=\bar{1}$. We omit $\otimes$ in $\cA$ for convenience. For any $i_1,\dots,i_k\in\overline{1,N}$, write $\xi_{i_1,\dots,i_k}:=\xi_{i_1}\cdots \xi_{i_k}$. Also, for any subset $I=\{i_1,\dots,i_k\} \subset \overline{1,N}$, write $\underline{I}=(l_1,\dots,l_k)$ if $\{l_1,\dots,l_k\}=\{i_1,\dots,i_k\}$ and $l_1<\dots<l_k$. Denote $\xi_I:=\xi_{l_1,\dots,l_k}$ and set $\xi_\varnothing=1$. Denote $\tau(i_1,\cdots,i_k)$ the inverse order of the sequence $i_1,\cdots,i_k$, and $\tau(I,J)=\tau(\underline{I},\underline{J})=\tau(k_1,\cdots,k_s,l_1,\cdots,l_r)$ when $I\cap J=\emptyset$ with $\underline{I}=(k_1,\cdots,k_s),\underline{J}=(l_1,\cdots,l_r)$. We set $\tau(\emptyset,\emptyset)=\tau(\emptyset)=0$. Denote $\xi_{I,J}=\xi_I\xi_J, \xi_{i,I}=\xi_i\xi_I,\xi_I\xi_i=\xi_{I,i}$.

Denote by $\cW(N)$ the superalgebra of all derivations of $\cA$. For $\epsilon=(\epsilon_1,\epsilon_2,\cdots,\epsilon_N)\in\bZ^N$, let $\omega_\epsilon:=dt-\sum\limits_{i=1}^Nt^{\epsilon_i}\xi_id\xi_i$, which is called a \emph{contact form} \cite{KL,Ka1,GLS}. The contact superconformal algebras  are defined as
\[
\cK(N;\epsilon)=\{D\in W(N)\,|\,D\omega_\epsilon=P\omega_\epsilon\mbox{ for some } P\in\cA\}.
\]
%By making a change of variables $\xi_i\mapsto t^{k_i}\xi_i$, we may assume that $\epsilon_i\in\{0,1\}$ for all $i$.
Note that $\cK(N;\epsilon)$ is simple except when $N=4$ and $\frac{1}{2}(\epsilon_1+\epsilon_2+\epsilon_3+\epsilon_4)\in\bZ$. Set $\cK'(N;\epsilon)=[\cK(N;\epsilon),\cK(N;\epsilon)]$. The following result was given in \cite{KL}.

\begin{theorem}
Let $N\ge1$.
\begin{enumerate}
\item There are only two non-isomorphic superalgebras $\cK(N;\epsilon)$ corresponding to $\epsilon=\zero$ and $(1,0,\ldots,0)$.
\item $\cK'(N;\epsilon)\cong \cK'(N;\epsilon')$ if and only if $\sum\limits_{i=1}^N(\epsilon_i-\epsilon'_i)\in2\bZ$.
\item $\cK'(N;\epsilon)$ has no nontrivial $2$-cocycles for $N>4$.
\end{enumerate}
\end{theorem}

Write $\partial_t:=\frac{\partial}{\partial t}, \partial_i:=\frac{\partial}{\partial\xi_i}$. There is one-to-one correspondence between the differential operators in $\cK(N;\epsilon)$ and the functions in $\cA$. The correspondence $f\leftrightarrow D_f$ is given by
\[
D_f:=(\Delta f)D^\epsilon+\sum\limits_{i=1}^ND^\epsilon(f)\xi_i\partial_i+(-1)^{|f|}\sum\limits_{i=1}^Nt^{-\epsilon_i}\partial_i(f)\partial_i,
\]
where $\Delta f:=2f-\sum\limits_{i=1}^N\xi_i\partial_i(f), D^\epsilon:=\partial_t-\frac{1}{2}t^{-1}\sum\limits_{i=1}^N\epsilon_i\xi_i\partial_i$. In particular, $\cK(N;\epsilon)$ has a $\bC$-basis $\{D_{t^k\xi_I}\,|\,k\in\bZ, I\subseteq\overline{1,N}\}$ with Lie brackets identified with the contact bracket in $\cA$:
\[
\{f,g\}:=(\Delta f)D^\epsilon(g)-D^\epsilon(f)(\Delta g)+(-1)^{|f|}\sum\limits_{i=1}^Nt^{-\epsilon_i}\partial_i(f)\partial_i(g),
\]
so that $[D_f,D_g]=D_{\{f,g\}}$. More precisely, we have
\begin{equation}\label{bracket1}
\begin{aligned}
&[D_{t^{k+1}\xi_I},D_{t^{l+1}\xi_J}]\\
=&\begin{cases}
0, & |I\cap J|>1,\\
(-1)^{|I|}D_{t^{k+l+2-\epsilon_i}\partial_i(\xi_I)\partial_i(\xi_J)}, & I\cap J=\{i\},\\
\big((2-|I|)(l+1-\frac{1}{2}\sum\limits_{j\in J}\epsilon_j)-(2-|J|)(k+1-\frac{1}{2}\sum\limits_{i\in I}\epsilon_i)\big)D_{t^{k+l+1}\xi_{I,J}}, & I\cap J=\emptyset.
\end{cases}
\end{aligned}
\end{equation}
$\cK(N;\epsilon)$ acts naturally on $\cA$, hence we have the extended algebras $\widetilde{\cK}(N;\epsilon):=\cK(N;\epsilon)\ltimes\cA$.

On the other hand, $\cK(N;\epsilon)$ is naturally a module over the abelian superalgebra $\cA$ with the actions $fD_g:=D_{fg}$. And together with with adjoint $\cK(N;\epsilon)$ actions, $\cK(N;\epsilon)$ is a $\widetilde{\cK}(N;\epsilon)$ module. To see this, we only need to verify
\[
[D_f,hD_g]-(-1)^{|h||f|}h[D_f,D_g]=D_f(h)D_g=D_{D_f(h)g},
\]
that is,
\[
\{f,hg\}-(-1)^{|h||f|}h\{f,g\}=D_f(h)g.
\]
The formula follows from the fact that $\{\,,\,\}$ is a contact bracket for $\cA$ (\cite{MZe1}, Example 12).

For the rest of this paper, we assume that $N\ne4$. Denote by $\widehat{\cK}(N;\epsilon)$ the universal central extension of $\cK(N;\epsilon)$. For $N>4, \widehat{\cK}(N;\epsilon)=\cK(N;\epsilon)$. For $N<4$, $\widehat{\cK}(N;\epsilon)$ has a $\bC$-basis $\{C,D_{t^k\xi_I}\,|\,k\in\bZ, I\subseteq\overline{1,N}\}$ with brackets (see \cite{KL})
\begin{align*}
&[D_{t^{m+1}\xi_I},D_{t^{n+1}\xi_J}]\\
=&D_{\{t^{m+1}\xi_I,t^{n+1}\xi_J\}}+\delta_{I,J}\delta_{m+n+\sum\limits_{i\in I}(1-\epsilon_i),0}\frac{C}{3}\Big(\delta_{|I|,0}m(m^2-1)\\
&+\delta_{|I|,1}((m+\frac{1}{2}\sum\limits_{i\in I}(1-\epsilon_i))^2-\frac{1}{4})+\delta_{|I|,2}(m+\frac{1}{2}\sum\limits_{i\in I}(1-\epsilon_i))+\delta_{|I|,3}\Big).
\end{align*}

An $\cA\cK(N;\epsilon)$ module is a $\widetilde{\cK}(N;\epsilon)$ module with $\cA$  acting associatively.
Let $\fg$ be any subalgebra of $\widetilde{\cK}(N;\epsilon)$ containing $D_t$ or $\widehat{\cK}(N;\epsilon)$. Denote $\sW(\fg)$ be the full category consisting of $\fg$ modules on which the action of $D_t$ is diagonalizable.  Let $M\in\sW(\fg)$. Then $M=\bigoplus\limits_{\lambda\in\bC}M_\lambda$, where $M_\lambda=\{v\in M\,|\,D_{t}v=\lambda v\}$, called the \emph{weight space} of weight $\lambda$. $\Supp(M):=\{\lambda\in\bC\,|\,M_\lambda\neq0\}$ is called the \emph{support} of $M$. A module $M\in\sW(\fg)$ is called a \emph{quasifinite module} if $\dim M_\lambda<\infty$ for all $\lambda$.  A module $M\in\sW(\fg)$ is called a \emph{bounded module} (or a cuspidal module) if the dimensions of weight spaces are bounded, that is there is $b\in\bN$ with $\dim M_\lambda<b$ for all $\lambda\in\Supp(M)$. Clearly, if $M$ is simple, then $\Supp(M)\subseteq\lambda+(1+\delta_{\epsilon,\one})\bZ$ for some $\lambda\in\bC$. Denote by $\sF(\fg)$ ($\sB(\fg)$, respectively) the full subcategory of $\sW(\fg)$ consisting of quasifinite (bounded, respectively) $\fg$ modules.

A \emph{jet $\cK(N;\epsilon)$ module} associated to $\cA$ is an  $\cA\cK(N;\epsilon)$ module in $\sF(\widetilde{\cK}(N;\epsilon))$. Denote the category consisting of all jet $\cK(N;\epsilon)$ modules associated to $\cA$ by $\sJ(\cK(N;\epsilon),\cA)$. Clearly, any simple module in $\sJ(\cK(N;\epsilon),\cA)$ is bounded. For a fixed $\lambda\in\bC$, denote by $\sJ_\lambda(K(N;\epsilon),A)$ the subcategory of $\sJ(\cK(N;\epsilon),\cA)$ supported on $\lambda+(1+\delta_{\epsilon,\one})\bZ$.

The subalgebra of $\cK(N;\epsilon)$ spanned by $\{D_{t^{k+1}}\,|\, k\in\bZ\}$ is isomorphic to the Witt algebra $\cW(0)$.  The following results for $\cW(0)$ modules will be used.

\begin{lemma}[{\cite[Lemma 2.6]{CLW}}]\label{cofinite}
Any co-finite ideal of $(t-1)\cW(0)$ contains $(t-1)^\ell \cW(0)$ for large $\ell$.
\end{lemma}

We also need the following results on $\cK(N;\epsilon)$.

\begin{lemma}\label{equivcat}
If $\cK(N;\epsilon)\cong\cK(N;\epsilon')$, then the category $\sB(\cK(N;\epsilon))$ is naturally equivalent to $\sB(\cK(N;\epsilon'))$.
\end{lemma}
\begin{proof}
If $\epsilon-\epsilon'=(2k_1,\cdots,2k_N)\in(2\bZ)^N$, then there exists an isomorphism from $\cK(N;\epsilon)$ to $\cK(N;\epsilon')$ mapping $D_{t^k\xi_I}$ to $D_{t^{k+\sum_{i\in I}k_i}\xi_I}$. The isomorphism preserves $D_t$ and therefore $\sB(\cK(N;\epsilon))$ is naturally equivalent to $\sB(\cK(N;\epsilon'))$. Now we assume $\epsilon,\epsilon'\in\{0,1\}^N$. Without loss of generality, we may assume that $\epsilon'-\epsilon=(1,1,0,\cdots,0)$. Recall from \cite{KL}, there is an isomorphism $\sigma$ from $\cK(N;\epsilon)$ to $\cK(N;\epsilon')$ mapping $D_t$ to $D_t+\boldsymbol{i}D_{t\xi_1\xi_2}$. Any $\cK(N;\epsilon')$ module $M$ can be viewed as a $\cK(N;\epsilon)$ module via $x\cdot v=:\sigma(x)v$. Hence, we have a functor $\cT_\sigma$ from $\cK(N;\epsilon')$-$\mathrm{Mod}$ to $\cK(N;\epsilon)$-$\mathrm{Mod}$. Similarly, we have $\cT_{\sigma^{-1}}$ and $\cT_{\sigma}\cT_{\sigma^{-1}}=\cT_{\sigma^{-1}}\cT_{\sigma}\xrightarrow{\sim}\mathrm{Id}$. We claim that $\cT_\sigma$ provides an equivalence from $\sB(\cK(N;\epsilon))$ to $\sB(\cK(N;\epsilon'))$.
%\begin{align*}
%&\cT_\sigma(V)\in\sB(\cK(N;\epsilon)), \forall V\in\sB(\cK(N;\epsilon'));\\
%&\cT_\sigma(U)\in\sB(\cK(N;\epsilon')), \forall U\in\sB(\cK(N;\epsilon)).
%\end{align*}
%We only prove the first statement, the proof of the second one is similar.

Note that any $V\in\sB(\cK(N;\epsilon'))$ is a bounded $\sum\limits_{k\in\bZ}(\bC D_{t^k}+\bC D_{t^k\xi_1\xi_2})$ module and via the adjoint action, $D_{t\xi_1\xi_2}$ acts diagonally on $\cK(N;\epsilon')$. Since $\sum\limits_{k\in\bZ}(\bC D_{t^k}+\bC D_{t^k\xi_1\xi_2})$ is isomorphic to the twisted Heisenberg-Virasoro algebra, following from \cite{LuZ}, $V=\bigoplus_{\lambda\in\bZ,\mu\in S}V_{\lambda,\mu}$ for some finite subset $S$ of $\bZ$, here $V_{\lambda,\mu}=\{v\in V\,|\,D_t v=\lambda v, D_{t\xi_1\xi_2}v=\mu v\}$. So, $\cT_\sigma(V)_\gamma=\sum\limits_{\mu\in S}V_{\gamma-\mu,\mu}$, is bounded, that is $\cT_\sigma(V)\in\sB(\cK(N;\epsilon))$.
\end{proof}

\begin{lemma}\label{omegaoper}
Assume that $N\ne 4$. Any module in $\sB(\cK(N;\epsilon))$ is annihilated by \[
\Omega_{m,k,p,I}:=\sum\limits_{i=0}^m(-1)^i\binom{m}{i}D_{t^{k-i}\xi_I}D_{t^{p+i}}\]
for some $m\in\bN$ and for all $k,p\in\bZ, I\subseteq\overline{1,N}$.
\end{lemma}
\begin{proof}
Let $V\in\sB(\cK(N;\epsilon))$. For $I=\emptyset$, regard $V$ as a bounded $\cW(0)$ module, then Corollary 3.4 in \cite{BF} tells us that $V$ is annihilated by $\Omega_{m,k,p,\emptyset}$ for some $m$ and for all $k,p$.

For $|I|>1$ with $|I|\neq4$, regard $V$ as a bounded $\cW(0)+\sum\limits_{k\in\bZ}\bC D_{t^k\xi_I}$ module, then Lemma 4.4 in \cite{CLW} tells us that $V$ is annihilated by $\Omega_{m,k,p,I}$ for some $m$ and for all $k,p$.

For $|I|=1$, the subalgebra $\cW(0)+\sum\limits_{k\in\bZ}\bC D_{t^k\xi_j}$ is isomorphic to $K(1;\epsilon_j)$. Then from Lemma 3.7 in \cite{CLL} and Lemma 4.3 in \cite{CL}, we know that $V$ is annihilated by $\Omega_{m,k,p,\{j\}}$ for some $m$ and for all $k,p$.

To finish the proof, for $N\ge5$ and $|I|=4$, take $\ell\notin I$. Without loss of generality, we may assume $\tau(I,\{\ell\})=0$. On $V$ we have
{\footnotesize\begin{align*}
&0=[\Omega_{m,k,p,\{\ell\}},D_{t^j\xi_{I,\ell}}]-2[\Omega_{m,k,p-1,\{\ell\}},D_{t^{j+1}\xi_{I,\ell}}]+[\Omega_{m,k,p-2,\{\ell\}},D_{t^{j+2}\xi_{I,\ell}}]\\
&=\sum\limits_{i=0}^m(-1)^i\binom{m}{i}\big(-D_{t^{k+j-i-\epsilon_\ell}\xi_I}D_{t^{p+i}}+(2j-\sum\limits_{a\in I\cup\{\ell\}}\epsilon_{a}+3p+3i)D_{t^{k-i}\xi_\ell}D_{t^{j+p+i-1}\xi_{I,\ell}}\big)\\
&-2\sum\limits_{i=0}^m(-1)^i\binom{m}{i}\big(-D_{t^{k+j+1-i-\epsilon_\ell}\xi_I}D_{t^{p-1+i}}+(2j-\hskip-0.3cm\sum\limits_{a\in I\cup\{\ell\}}\hskip-0.3cm\epsilon_{a}+3p+3i-1)D_{t^{k-i}\xi_\ell}D_{t^{j+p+i-1}\xi_{I,\ell}}\big)\\
&+\sum\limits_{i=0}^m(-1)^i\binom{m}{i}\big(-D_{t^{k+j+2-i-\epsilon_\ell}\xi_I}D_{t^{p-2+i}}+(2j-\hskip-0.3cm\sum\limits_{a\in I\cup\{\ell\}}\hskip-0.3cm\epsilon_{a}+3p+3i-2)D_{t^{k-i}\xi_\ell}D_{t^{j+p+i-1}\xi_{I,\ell}}\big)\\
&=-\sum\limits_{i=0}^{m+2}(-1)^i\binom{m+2}{i}D_{t^{k+j-i-\epsilon_\ell}\xi_I}D_{t^{p+i}}.\qedhere
\end{align*}}
\end{proof}

\begin{lemma}\label{algisom}
The map $\Phi: \widetilde{\cK}(N;\epsilon)\to\widetilde{\cK}(N;\epsilon)$ defined by mapping $D_{t^{k+1}\xi_I}$ to $D_{t^{k+1}\xi_I}+(2k+\sum\limits_{i\in I}(1-\epsilon_i))t^k\xi_I$ and mapping $t^k\xi_I$ to $t^k\xi_I$ for all $k\in\bZ, I\subseteq\overline{1,N}$ is an automorphism.
\end{lemma}
\begin{proof}
Lemma follows from direct computations.
\end{proof}

\begin{lemma}\label{ideal}
For $\ell\in\bZ_+$, let $\cK_\ell:=\sum\limits_{I\subseteq\overline{1,N}}(t-1)^{\max\{\ell+1-\lfloor\frac{|I|+\delta_{\ell,0}}{3}\rfloor,0\}}\xi_I\cK(N;\epsilon)$.
\begin{enumerate}
\item $\cK_0$ is a subalgebra of $\cK(N;\epsilon)$.
\item $\cK_\ell$ is an ideal of $\cK_0$.
\item $[\cK_1,\cK_\ell]\subseteq\cK_{\ell+1}$.
\item The ideal of $\cK_0$ generated by $(t-1)^{\ell+1} \cW(0)$ contains $\cK_{\ell+1}$.
\item $\mathfrak{i}:=(t-1)^2\cK(N;\epsilon)+\sum\limits_{1\le i\le N}(t-1)\xi_i\cK(N;\epsilon)+\sum\limits_{1\le i<j\le N}(t-1)\xi_i\xi_j\cK(N;\epsilon)+\sum\limits_{I\subseteq\overline{1,N},|I|\ge3}\xi_I\cK(N;\epsilon)$ is an ideal of $\cK_0$ containing $\cK_1$.
\end{enumerate}
\end{lemma}
\begin{proof}
From (\ref{bracket1}), we have: for $\ell_1,\ell_2\in\bZ_+$,
\begin{align*}
&[(t-1)^{\ell_1}D_{t^{i}\xi_I},(t-1)^{\ell_2}D_{t^{j}\xi_J}]\\
=&\begin{cases}
0, & |I\cap J|>1,\\
(-1)^{|I|}(t-1)^{\ell_1+\ell_2}D_{t^{i+j-\epsilon_i}\partial_i(\xi_I)\partial_i(\xi_J)}, & I\cap J=\{i\},\\
\big((2-|I|)(j-\frac{1}{2}\sum\limits_{s\in J}\epsilon_s)-(2-|J|)(i-\frac{1}{2}\sum\limits_{s\in I}\epsilon_s)\big)(t-1)^{\ell_1+\ell_2}D_{t^{i+j-1}\xi_{I,J}}&\\
+\big(\ell_2(2-|I|)-\ell_1(2-|J|)\big)(t-1)^{\ell_1+\ell_2-1}D_{t^{i+j}\xi_{I,J}}, & I\cap J=\emptyset.
\end{cases}
\end{align*}
So lemma follows.
\end{proof}

\begin{proposition}\label{ck1}
For any finite dimensional simple $\cK_0$ module $V$, we have $\cK_1V=0$. Hence, $V$ is a simple module over the finite dimensional Lie superalgebra $\cK_0/\cK_1$.
\end{proposition}
\begin{proof}
From Lemma \ref{cofinite}, we know $V$ is annihilated by $(t-1)^\ell\cW(0)$ for some $\ell$. Hence, by Lemma \ref{ideal}, $\cK_\ell\subseteq\mathrm{ann}(V)$. Here, $\mathrm{ann}(V)$ is the ideal of $\cK_0$ that annihilates $V$. We are going to show $\cK_1\subseteq\mathrm{ann}(V)$. Consider the finite dimensional Lie superalgebra $\fg=\cK_0/\mathrm{ann}(V)$, then $V$ is a finite dimensional $\fg_{\bar{0}}$ module on which $\cK_{1,\bar{0}}+\mathrm{ann}(V)$ acts nilpotently. Since $[x,x]\in\cK_{1,\bar{0}}$ for all $x\in\cK_{1,\bar{1}}$, every element in $\cK_{1,\bar{1}}+\mathrm{ann}(V)$ acts nilpotently on $V$. Hence, by Engel's Theorem for Lie superalgebra (Theorem 2.1 in \cite{M}), there is a nonzero $v\in V$ annihilated by $\cK_1+\mathrm{ann}(V)$. Therefore, $\cK_1V=0$.
\end{proof}

\begin{lemma}\label{cat-so}
We have $\cK_0/\mathfrak{i}\cong\bC z\oplus\fs\fo_N$. Moreover, the category of simple finite dimensional module over $\cK_0$ is naturally equivalent to the category of simple finite dimensional module over $\bC z\oplus\fs\fo_N$.
%The category of simple finite dimensional module over $\cK_0/\cK_1$ is naturally equivalent to the category of simple finite dimensional module over $\bC z\oplus\mathfrak{so}_n$.
\end{lemma}
\begin{proof}
The map $D_t-D_1+\mathfrak{i}\mapsto z, D_{t\xi_i\xi_j}+\mathfrak{i}\mapsto E_{ji}-E_{ij}$ provides an isomorphism from $\cK_0/\mathfrak{i}$ to $\bC z\oplus\fs\fo_N$. And therefore any $\bC z\oplus\fs\fo_N$ module is naturally a $\cK_0$ module. Conversely, Proposition \ref{ck1} tells us that any simple finite dimensional $\cK_0$ module is a simple $\cK_0/\cK_1$ module. Then from
\begin{align*}
&[D_{t-1}+\cK_1,D_{(t-1)\xi_i}+\cK_1]=D_{(t-1)\xi_i}+\cK_1,\\
&[D_{t-1}+\cK_1,D_{t\xi_I}+\cK_1]=(2-|I|)D_{t\xi_I}+\cK_1, |I|=3,4,5,
\end{align*}
we know that $\mathfrak{i}$ acts trivially on any simple finite dimensional $\cK_0$ module.
\end{proof}

\section{Jet modules}
\label{cuspidalAK}

In this section, we will classify all simple jet modules in $\sJ(\cK(N;\epsilon), \cA)$.

\subsection{The construction of tensor modules}

We will apply the weighting functor introduced in \cite{N} to construct ``tensor modules" for $\widetilde{\cK}(N;\epsilon)$.  Let $\fg$ be $\cK(N;\epsilon)$ or $\cL:=\sum\limits_{|I|\ne1, k\in\bZ}\bC D_{t^k\xi_I}$. For any $\cA\fg$ module $P$ and $a\in\bC$, denote
\[
\fW^{(a)}(P):=\bigoplus\limits_{i\in(1+\delta_{\epsilon,\one})\bZ}(P/(D_t-a-i)P\otimes t^{a+i}).
\]
By Proposition 8 in \cite{N}, we know that $\fW^{(a)}(P)$ is an $\cA\fg$ module with
\[
x\cdot((v+(D_t-a-j)P)\otimes t^{a+j}):=(xv+(D_t-a-j-r)P)\otimes t^{a+j+r},
\]
for all $x\in(\fg\ltimes\cA)_r, v\in P, j\in(1+\delta_{\epsilon,\one})\bZ$.

\begin{lemma}\label{weighting}
Let $\fg$ be $\cK(N;\epsilon)$ or $\cL$. For any $M\in\sJ_a(\fg,\cA)$, we have $M\cong\fW^{(a)}(M)$ as $\cA\fg$ modules.
\end{lemma}
\begin{proof}
Consider the map $\varphi: M\to\fW^{a}(M); v\mapsto\overline{v}\otimes t^{a+l}, \forall v\in M_{a+l}$. Then $\varphi$ is an $\cA\fg$ module homomorphism since
\[
\varphi(xv)=\overline{xv}\otimes t^{a+l+r}=x(\overline{v}\otimes t^{a+l}), \forall x\in(\fg\ltimes \cA)_r.
\]
From the facts that $\dim M_{a+l}=\dim\fW^{(a)}(M)_{a+l}$ and $\varphi$ is injective, we know that $\varphi$ is an isomorphism.
\end{proof}

By Lemma \ref{cat-so}, any  $\bC z\oplus\fs\fo_N$ module $V$ with $z$ acting as a scalar $c$ is naturally a $\cK_0$ module. Then for $a\in\bC$, $\cF_a(V,c):=\fW^{(a)}(\mathrm{Ind}_{\cK_0\ltimes\cA}^{\widetilde{\cK}(N;\epsilon)}V)=\oplus_{I\subseteq\overline{1,N}}\mathbf{D}_IV\otimes t^a\bC[t^{1+\delta_{\epsilon,\one}},t^{-1-\delta_{\epsilon,\one}}]$ is a $\widetilde{\cK}(N;\epsilon)$ module under the following actions
\begin{align}
t^k\xi_I.\mathbf{D}_Jv\otimes t^{a+l}&=\delta_{|I\setminus J|,0}(-1)^{\tau(I,J\setminus I)+\frac{|I|(|I|+1)}{2}}\mathbf{D}_{J\setminus I}v\otimes t^{a+l+2k+\sum\limits_{i\in I}(1-\epsilon_i)},\\
D_{t^{k+1}\xi_I}.v\otimes t^{a+l}&=\begin{cases}
(a+l+2k+ck)v\otimes t^{a+l+2k}, & I=\emptyset,\\
D_{t\xi_i}v\otimes t^{a+l+2k+1-\epsilon_i}, & I=\{i\},\\
0, & |I|>2,\\
(E_{sr}-E_{rs})v\otimes t^{a+l+2k+2-\epsilon_s-\epsilon_r}, & I=\{r,s\}, r<s;
\end{cases}\\
D_{t^{k+1}\xi_I}.\mathbf{D}_J\bar{v}\otimes t^{a+l}&=[D_{t^{k+1}\xi_I},D_{t\xi_{j_1}}]\mathbf{D}_{J\setminus \{j_1\}}\overline{v}\otimes t^{a+l-(1-\epsilon_{j_1})}\nonumber\\
&+(-1)^{|I|}D_{t\xi_{j_1}}D_{t^{k+1}\xi_I}\mathbf{D}_{J\setminus\{j_1\}}\overline{v}\otimes t^{a+l-(1-\epsilon_{j_1})},
\end{align}
here for $I=\{i_1,\cdots,i_s\}\subseteq\overline{1,N}$ with $i_1<\cdots<i_s$, denote by $\mathbf{D}_I:=D_{t\xi_{i_1}}D_{t\xi_{i_2}}\cdots D_{t\xi_{i_s}}$. Moreover, $\cF_a(V,c)$ belongs to $\sJ_a(\cK(N;\epsilon),\cA)$ if and only if $V$ is finite dimensional.

\begin{proposition}
Let $V$ be a simple finite dimensional $\bC z\oplus\fs\fo_N$ module with $z$ acting as $c\in\bC$ and let $a\in\bC$.
\begin{enumerate}
\item $\cF_a(V,c)$ is a simple module in $\sJ_a(\cK(N;\one),\cA)$.
\item If $\epsilon=\zero$, let $\cF_a(V,c)_\delta:=\bigoplus\limits_{J\subseteq\overline{1,N}}\mathbf{D}_JV\otimes t^{a+\delta+|J|}\bC[t^2,t^{-2}]$, where $\delta=0$ or $1$. Then $\cF_a(V,c)_\delta$ are simple $\cA\cK(N;\zero)$ submodules of $\cF_a(V,c)$ and $\cF_a(V,c)=\cF_a(V,c)_0\oplus\cF_a(V,c)_1$.
\end{enumerate}
\end{proposition}
\begin{proof}
\begin{enumerate}
\item Let $W$ be a nonzero submodule of $\cF_a(V,c)$. Since $\cF_a(V,c)$ is a weight module, we may assume that $W$ contains a nonzero element of the form $\sum\limits_{J\subseteq\overline{1,N}}D_Jv_J\otimes t^{a+l}$. Applying $\xi_J$, we know that $W$ contains a nonzero element of the form $v\otimes t^{a+l}$ for some $v\in V$ and some $l\in\bZ$. With the actions of $D_{t^{k+1}\xi_i\xi_j}$, we get $(E_{ji}-E_{ij})v\otimes t^{a+l+2k}\in W$ for all $i,j\in\overline{1,N}$ and $k\in\bZ$. Finally, from the actions of $D_{t\xi_i}$ ($1\le i\le N$), we get $W=\cF_a(V,c)$.
\item It is easy to check that $\cF_a(V,c)_\delta$ are $\cA\cK(N;\zero)$ submodules when $\epsilon=\zero$. Using the same argument as 1, we deduce that $\cF_a(V,c)_\delta$ is simple. The fact $\cF_a(V,c)=\cF_a(V,c)_0\oplus\cF_a(V,c)_1$ is clear.\qedhere
\end{enumerate}
\end{proof}

\begin{proposition}\label{isomtensor}
Let $a,b\in\bC$ and let $V,V'$ be a simple finite dimensional $\bC z\oplus\fs\fo_N$ module with $z$ acting as $c,c'\in\bC$, respectively.
\begin{enumerate}
\item For $\epsilon=\one$, $\cF_a(V,c)\cong\cF_b(V',c')$ as $\cA\cK(N;\one)$ modules if and only if $V\cong V'$ as $\bC z\oplus\fs\fo_N$ modules (from which one deduce that $c=c'$) and $a-b\in2\bZ$.
\item Suppose $\epsilon=\zero$.
\begin{enumerate}
\item As $\cA\cK(N;\zero)$ modules, we have $\cF_a(V,c)_1\cong\cF_{a+1}(V,c)_0$.
\item As $\cA\cK(N;\zero)$ modules, we have $\cF_a(V,c)_0\cong\cF_b(V',c')_0$ if and only if $V\cong V'$ as $\bC z\oplus\fs\fo_N$ modules and $a-b\in2\bZ$.
\end{enumerate}
\end{enumerate}
\end{proposition}
\begin{proof}
\begin{enumerate}
\item The sufficiency is obvious. For the necessity, suppose $\tilde{\gamma}: \cF_a(V,c)\xrightarrow{\sim}\cF_b(V',c')$. Then $\Supp(\cF_a(V,c))=\Supp(\cF_b(V',c'))$. Hence, $a-b\in2\bZ$. Let $v\in V$ be nonzero. From the action of $D_t$, we may assume that $\tilde{\gamma}(v\otimes t^a)=v'\otimes t^a$ for some nonzero $v'\in V'$. Then
    \[
    \tilde{\gamma}(v\otimes t^{a+2k})=\tilde{\gamma}(t^k\cdot(v\otimes t^a))=t^k\cdot\tilde{\gamma}(v\otimes t^a)=t^k\cdot(v'\otimes t^a)=v'\otimes t^{a+2k}, \forall k\in\bZ.
    \]
    Therefore,
    \begin{align*}
    (a+(c+2)k)v'\otimes t^{a+2k}&=\tilde{\gamma}(D_{t^{k+1}}\cdot(v\otimes t^a))\\
    &=D_{t^{k+1}}\cdot\tilde{\gamma}(v\otimes t^{a})=(a+(c'+2)k)v'\otimes t^{a+2k}, \forall k\in\bZ.
    \end{align*}
    So, $c=c'$. Thus, statement is true for $N=1$. For $N\ge 2$, we have $\tilde{\gamma}(V\otimes t^a)=V'\otimes t^a$. Then from the actions of $D_{t\xi_i\xi_j}$, one can deduce that $\tilde{\gamma}|_{V\otimes t^a}$ induces an $\fs\fo_N$ module isomorphism from $V$ to $V'$.
\item The proof of statement (b) is similar as the proof of statement 1. For statement (a), it is easy to check that the map from $\cF_a(V,c)_1$ to $\cF_{a+1}(V,c)_0$ defined by mapping $\mathbf{D}_Jv\otimes t^{a+1+|J|+2k}$ to $\mathbf{D}_Jv\otimes t^{a+1+|J|+2k}$ is an $\cA\cK(N;\zero)$ module isomorphism.\qedhere
\end{enumerate}
\end{proof}

\subsection{Classification of simple Jet modules}

Let $a\in\bC$ and $M\in\sJ_a(\cK(N;\epsilon),\cA)$ be nonzero and simple. We claim that $M':=\{v\in M\,|\,\xi_iv=0, \forall i\in\overline{1,N}\}$ is an $\cL\ltimes\cA$ module. Clearly, we have $\cA M'\subseteq M'$. Hence, claim follows from
\[
[D_{t^{k+1}\xi_I},\xi_i]=\delta_{i\notin I}(k+1-\frac{1}{2}\sum\limits_{j\in I}\epsilon_j+\frac{1}{2}(|I|-2)\epsilon_i)t^k\xi_{I,i}+\delta_{i\in I}(-1)^{|I|}t^{k+1-\epsilon_i}\partial_i(\xi_I).
\]
Let $\cL':=(t-1)\sum\limits_{k\in\bZ}\bC D_{t^k}+\sum\limits_{k\in\bZ, |I|>1}\bC D_{t^k\xi_I}$. Then $(t-1)M'$ is an $\cL'\ltimes\cA$ module since $[D_{t^{k+1}\xi_I},t-1]=(2-|I|)t^{k+1}\xi_I$. We will classify simple modules in $\sJ_a(\cK(N;\epsilon),\cA)$ by showing the following diagram commutes for $M'/(t-1)M'$:
\[
\xymatrix{\cL'\ltimes\cA\mathrm{-Mod}\ar[r]^{\mathrm{Ind}}\ar[d]^{\mathrm{Ind}} & \cL\ltimes\cA\mathrm{-Mod}\ar[r]^{\fW^{(a)}} &\cL\times\cA\mathrm{-Mod}\ar[d]^{\mathrm{Ind}}\\
\cK_0\ltimes\cA\mathrm{-Mod}\ar[r]^{\mathrm{Ind}} & \cK\ltimes\cA\mathrm{-Mod}\ar[r]^{\fW^{(a)}} & \cK\ltimes\cA\mathrm{-Mod}
}
\]

Indeed, we will prove
\begin{theorem}\label{jetmodule}
Any simple module in $\sJ_a(\cK(N;\epsilon),\cA)$ is isomorphic to some simple quotient of some $\cF_a(V,c)$ for some simple finite dimensional $\bC z\oplus\fs\fo_N$ module $V$ with $z$ acting as $c\in\bC$.
\end{theorem}

Note that $M'/(t-1)M'$ is a finite dimensional $\cL'\ltimes\cA$ module. Consider the $\cL\ltimes\cA$ module
\[
\widetilde{M}:=\mathrm{Ind}_{\cL'\ltimes\cA}^{\cL\ltimes\cA}M'/(t-1)M'=\bC[D_t]\otimes M'/(t-1)M'.
\]
Applying the weighting functor, via identify $M'/(t-1)M'$ with $\widetilde{M}/(D_t-a-j)\widetilde{M}$, we have $\fW^{(a)}(\widetilde{M})=M'/(t-1)M'\otimes t^a\bC[t^{1+\delta_{\epsilon,\one}}, t^{-1-\delta_{\epsilon,\one}}]$, with the actions
\begin{equation}\label{action1}
\begin{aligned}
t^k\xi_I.\bar{v}\otimes t^{a+l}&=\delta_{|I|,0}\bar{v}\otimes t^{a+l+2k},\\
D_{t^{k+1}\xi_I}.\bar{v}\otimes t^{a+l}&=\begin{cases}
((D_{t^{k+1}}-D_t)\bar{v}+(a+l+2k)\bar{v})\otimes t^{a+l+2k}, & I=\emptyset,\\
\overline{D_{t^{k+1}\xi_I}v}\otimes t^{a+l+2k+\sum\limits_{i\in I}(1-\epsilon_i)}, & |I|>1.
\end{cases}
\end{aligned}
\end{equation}
Clearly, for $\Phi$ defined in Lemma \ref{algisom}, $\Phi|_{\cL\ltimes\cA}$ is still an automorphism of $\cL\ltimes\cA$. For convenience, we still denote $\Phi|_{\cL\ltimes\cA}$ by $\Phi$.

\begin{lemma}\label{modisom1}
As $\cL\ltimes\cA$ modules, $\fW^{(a)}(\widetilde{M})$ is isomorphic to $(M')^\Phi$.
\end{lemma}
\begin{proof}
Since $M'$ is a submodule of the weight $\cL\ltimes\cA$ module $M$, $M'$ is also a weight module. Consider the map $\psi: (M')^\Phi\to\fW^{(a)}(\widetilde{M}); v\mapsto\overline{v}\otimes t^{a+l}$, where $v\in(M')_{a+l}$. Then lemma follows from the following computations: for $k\in\bZ, I,J\subseteq\overline{1,N}, |J|>1, v\in(M')_{a+l}$,
\begin{align*}
\psi(t^k\xi_I.v)&=\delta_{|I|,0}\overline{t^kv}\otimes t^{a+l+2k}=\delta_{|I|,0}\overline{v}\otimes t^{a+l+2k}=t^k\xi_I\cdot(\overline{v}\otimes t^{a+l})=t^k\xi_I\cdot\psi(v),\\
\psi(D_{t^{k+1}\xi_J}.v)&=\overline{D_{t^{k+1}\xi_j}v}\otimes t^{a+l+2k+\sum\limits_{j\in J}(1-\epsilon_j)}=D_{t^{k+1}\xi_J}\cdot\psi(v),\\
\psi(D_{t^{k+1}}.v)&=\overline{(D_{t^{k+1}}+2kt^k)v}\otimes t^{a+l+2k}=\overline{(D_{t^{k+1}}-D_t)v+(a+l+2k)v}\otimes t^{a+l+2k}\\
&=D_{t^{k+1}}\cdot\psi(v).\qedhere
\end{align*}
\end{proof}

Now let us consider the induced module $\mathrm{Ind}_{\cL\ltimes\cA}^{\widetilde{\cK}(N;\epsilon)}\widetilde{M}$. Clearly, we have
\[
\mathrm{Ind}_{\cL\ltimes\cA}^{\widetilde{\cK}(N;\epsilon)}\widetilde{M}=\mathrm{Ind}_{\cL\ltimes\cA}^{\widetilde{\cK}(N;\epsilon)}\mathrm{Ind}_{\cL'\ltimes\cA}^{\cL\ltimes\cA}M'/(t-1)M'=\mathrm{Ind}_{\cK_0\ltimes\cA}^{\widetilde{\cK}(N;\epsilon)}\mathrm{Ind}_{\cL'\ltimes\cA}^{\cK_0\ltimes\cA}M'/(t-1)M'.
\]
Denote $\mathrm{Ind}_{\cL'\ltimes\cA}^{\cK_0\ltimes\cA}M'/(t-1)M'$ by $\cM$. Then by induction we have
\[t^k\xi_Iv=\delta_{|I|,0}v, \forall v\in\cM, k\in\bZ,I\subseteq\overline{1,N}.\]
Indeed, for the induction step, for any $v=(t-1)D_{t^l\xi_j}u_1\cdots u_r\bar{w}\in\cM$ with $u_i\in\sum\limits_{k\in\bZ,j\in\overline{1,N}}\bC (t-1)D_{t^k\xi_j}$, we have
\begin{align*}
t^k\xi_Iv&=[t^k\xi_I,(t-1)D_{t^l\xi_j}]u_1\cdots u_r\bar{w}+(-1)^{|I|}(t-1)D_{t^l\xi_j}t^k\xi_Iu_1\cdots u_r\bar{w}\\
&=\delta_{I,\{j\}}(-t^{l+k+1-\epsilon_j}+t^{l+k-\epsilon_j})u_1\cdots u_r\bar{w}+\delta_{|I|,0}(t-1)D_{t^l\xi_j}u_1\cdots u_r\bar{w}\\
&=\delta_{|I|,0}(t-1)D_{t^l\xi_j}u_1\cdots u_r\bar{w}.
\end{align*}

Denote
\[\mathbb{M}:=\mathrm{Ind}_{\cK_0\ltimes\cA}^{\widetilde{\cK}(N;\epsilon)}\cM=\bC[D_t]\otimes(\oplus_{I\subseteq\overline{1,N}}\mathbf{D}_I \cM),\]
here for $I=\{i_1,\cdots,i_s\}\subseteq\overline{1,N}$ with $i_1<\cdots<i_s$, denote by $\mathbf{D}_I:=D_{t\xi_{i_1}}D_{t\xi_{i_2}}\cdots D_{t\xi_{i_s}}$. Note that $\bM$ is $\bC[D_t]$ free. We have a natural vector spaces isomorphism $\bM/(D_t-a-j)\bM\to \oplus_{I\subseteq\overline{1,N}}\mathbf{D}_I \cM$. Applying the weighting functor, via identify $\oplus_{I\subseteq\overline{1,N}}\mathbf{D}_I \cM$ with $\bM/(D_t-a-j)\bM$, we have
\[
\fW^{(a)}(\bM)=(\oplus_{I\subseteq\overline{1,N}}\mathbf{D}_I\cM)\otimes t^a\bC[t^{1+\delta_{\epsilon,\one}},t^{-1-\delta_{\epsilon,\one}}],
\]
with the actions
\begin{align}
t^k\xi_I.\mathbf{D}_Jv\otimes t^{a+l}&=\delta_{|I\setminus J|,0}(-1)^{\tau(I,J\setminus I)+\frac{|I|(|I|+1)}{2}}\mathbf{D}_{J\setminus I}v\otimes t^{a+l+2k+\sum\limits_{i\in I}(1-\epsilon_i)},\\
D_{t^{k+1}\xi_I}.v\otimes t^{a+l}&=\begin{cases}
((D_{t^{k+1}}-D_t)v+(a+l+2k)v)\otimes t^{a+l+2k}, & I=\emptyset,\\
((D_{t^{k+1}\xi_i}-D_{t\xi_i})v+D_{t\xi_i}v)\otimes t^{a+l+2k+1-\epsilon_i}, & I=\{i\},\\
D_{t^{k+1}\xi_I}v\otimes t^{a+l+2k+\sum\limits_{i\in I}(1-\epsilon_i)}, & |I|>1;
\end{cases}\\
D_{t^{k+1}\xi_I}.\mathbf{D}_J\bar{v}\otimes t^{a+l}&=[D_{t^{k+1}\xi_I},D_{t\xi_{j_1}}]\mathbf{D}_{J\setminus \{j_1\}}\overline{v}\otimes t^{a+l-(1-\epsilon_{j_1})}\nonumber\\
&+(-1)^{|I|}D_{t\xi_{j_1}}D_{t^{k+1}\xi_I}\mathbf{D}_{J\setminus\{j_1\}}\overline{v}\otimes t^{a+l-(1-\epsilon_{j_1})}.
\end{align}

It is easy to check that $\fW^{(a)}(\bM)$ is isomorphic to $\mathrm{Ind}_{\cL\ltimes\cA}^{\widetilde{\cK}(N;\epsilon)}\fW^{(a)}(\widetilde{M})$ via mapping $\overline{u_1\cdots u_r v}\otimes t^{a+l}$ to $u_1\cdots u_r(v\otimes t^{a+l-i_1-\cdots-i_r})$, where $u_j\in\{D_{t^k\xi_i}\,|\,k\in\bZ, i\in\overline{1,N}\}\cap\widetilde{\cK}(N;\epsilon)_{i_j}$. And hence, from Lemma \ref{modisom1}, we have $\fW^{(a)}(\bM)\cong\mathrm{Ind}_{\cL\ltimes\cA}^{\widetilde{\cK}(N;\epsilon)}(M')^\Phi$.

Thus, we have $\fW^{(a)}(\mathrm{Ind}_{\cL\ltimes\cA}^{\widetilde{\cK}(N;\epsilon)}\widetilde{M})\cong\mathrm{Ind}_{\cL\ltimes\cA}^{\widetilde{\cK}}(M')^\Phi$. Therefore, from the universal property of the functor $\mathrm{Ind}$, we get
\begin{proposition}
Let $a\in\bC$ and $M\in\sJ_a(\cK(N;\epsilon),\cA)$ be nonzero and simple. Then $M^\Phi$ is isomorphic to a simple quotient of $\fW^{(a)}(\bM)$.
\end{proposition}

Now we can prove Theorem \ref{jetmodule}
\begin{proof}[Proof of Theorem \ref{jetmodule}] Denote by the epimorphism from $\fW^{(a)}(\bM)$ to $M^\Phi$ by $\phi$. Then viewed $M$ as a $\cW(0)\ltimes \bC[t^{\pm1}]$ module, from Lemma 4.4 in \cite{CLW}, we know that there exists $\ell\in\bN$ such that $\sum\limits_{i=0}^\ell(-1)^i\binom{m}{i}t^{p-i}\cdot D_{t^{k+i}}\cdot  M=0$ for all $k,p\in\bZ$. Hence, on $M$, for all $p.k\in\bZ$, we have
\begin{align*}
&\sum\limits_{i=0}^\ell(-1)^i\binom{\ell}{i}D_{t^{k-i}}\cdot t^{p+i}=\sum\limits_{i=0}^\ell(-1)^i\binom{\ell}{i}([D_{t^{k-i}},t^{p+i}]+t^{p+i}\cdot D_{t^{k-i}})\\
=&2\sum\limits_{i=0}^\ell(-1)^i\binom{\ell}{i}(p+i)t^{k+p-1}+\sum\limits_{i=0}^\ell(-1)^{\ell-i}\binom{\ell}{i}t^{p+\ell-i}\cdot D_{t^{k-\ell+i}}=0.
\end{align*}
Therefore, for all $v\in\cM$ and $k\in\bZ$,
\begin{align*}
&\phi\Big((\sum\limits_{i=0}^\ell(-1)^i\binom{\ell}{i}D_{t^{k-i}}v)\otimes t^{a+l}\Big)=\phi\Big(\sum\limits_{i=0}^\ell(-1)^i\binom{\ell}{i}D_{t^{k-i}}\cdot(v\otimes t^{a+l-2k+2i})\Big)\\
=&\sum\limits_{i=0}^\ell(-1)^i\binom{\ell}{i}D_{t^{k-i}}\cdot\phi\Big(v\otimes t^{a+l-2k+2i}\Big)=\sum\limits_{i=0}^\ell(-1)^i\binom{\ell}{i}D_{t^{k-i}}\cdot\phi\Big(t^{-k+i}\cdot(v\otimes t^{a+l})\Big)\\
=&\sum\limits_{i=0}^\ell(-1)^i\binom{\ell}{i}D_{t^{k-i}}\cdot t^{-k+i}\cdot\phi\Big(v\otimes t^{a+l}\Big)=0.
\end{align*}
This means $((t-1)^\ell\cW(0)\cM)\otimes\bC[t^{1+\delta_{\epsilon,\one}},t^{-1-\delta_{\epsilon,\one}}]\subseteq\ker\phi$. Then by Lemma \ref{ideal}, $\cK_{\ell+1}\cM\otimes\bC[t^{1+\delta_{\epsilon,\one}},t^{-1-\delta_{\epsilon,\one}}]\subseteq\ker\phi$, and hence \[\fW^{(a)}(\mathrm{Ind}_{\cK_0\ltimes\cA}^{\widetilde{\cK}(N;\epsilon)}\cK_{\ell+1}\cM)=\oplus_{I\subseteq\overline{1,N}}\mathbf{D}_I\cK_{\ell+1}\cM\otimes\bC[t^{1+\delta_{\epsilon,\one}},t^{-1-\delta_{\epsilon,\one}}]\subseteq\ker\phi.\] Thus, $M$ is a simple quotient of $\fW^{(a)}(\mathrm{Ind}_{\cK_0\ltimes\cA}^{\widetilde{\cK}(N;\epsilon)}V)$%=\oplus_{I\subseteq\overline{1,N}}\mathbf{D}_IV\otimes t^a\bC[t^{1+\delta_{\epsilon,\one}},t^{-1-\delta_{\epsilon,\one}}]$
for some finite dimensional $\cK_0\ltimes\cA$ module $V$, on which $t^k\xi_I$ acts as $\delta_{|I|,0}\mathrm{id}$. Moreover, we may require $V$ to be simple. Indeed, since $V$ is finite dimensional, we have a composition series for $V$ as follows
 \[
 0=V_0\subset V_1\subset\cdots\subset V_k=V
 \]
 with $V_i/V_{i-1}$ being simple $\cK_0\ltimes\cA$ modules. Then we have the following composition series for $\fW^{(a)}(\mathrm{Ind}_{\cK_0\ltimes\cA}^{\widetilde{\cK}(N;\epsilon)}V)$:% $\oplus_{I\subseteq\overline{1,N}}\mathbf{D}_IV\otimes t^a\bC[t^{1+\delta_{\epsilon,\one}},t^{-1-\delta_{\epsilon,\one}}]$
\[
 0=\fW^{(a)}(\mathrm{Ind}_{\cK_0\ltimes\cA}^{\widetilde{\cK}(N;\epsilon)}V_0)\subseteq\fW^{(a)}(\mathrm{Ind}_{\cK_0\ltimes\cA}^{\widetilde{\cK}(N;\epsilon)}V_1)\subseteq\cdots\subseteq\fW^{(a)}(\mathrm{Ind}_{\cK_0\ltimes\cA}^{\widetilde{\cK}(N;\epsilon)}V_k). \]
Let $s$ be the minimal integer such that $\phi(\fW^{(a)}(\mathrm{Ind}_{\cK_0\ltimes\cA}^{\widetilde{\cK}(N;\epsilon)}V_s))\ne0$. Then $\phi(\fW^{(a)}(\mathrm{Ind}_{\cK_0\ltimes\cA}^{\widetilde{\cK}(N;\epsilon)}V_s/V_{s-1}))=M$. By Lemma \ref{cat-so}, $V$ is naturally a $\bC z\oplus\fs\fo_N$ module.
\end{proof}
\begin{comment}
Conversely, for any simple finite dimensional $\bC z\oplus\fs\fo_N$ module $V$ with $z$ acting as a scalar $c$, $\cF_a(V,c):=\fW^{(a)}(\mathrm{Ind}_{\cK_0\ltimes\cA}^{\widetilde{\cK}(N;\epsilon)}V)=\oplus_{I\subseteq\overline{1,N}}\mathbf{D}_IV\otimes t^a\bC[t^{1+\delta_{\epsilon,\one}},t^{-1-\delta_{\epsilon,\one}}]$ is a module in $\sJ_a(\cK(N;\epsilon),\cA)$ under the following actions
\begin{align}
t^k\xi_I.\mathbf{D}_Jv\otimes t^{a+l}&=\delta_{|I\setminus J|,0}(-1)^{\tau(I,J\setminus I)+\frac{|I|(|I|+1)}{2}}\mathbf{D}_{J\setminus I}v\otimes t^{a+l+2k+\sum\limits_{i\in I}(1-\epsilon_i)},\\
D_{t^{k+1}\xi_I}.v\otimes t^{a+l}&=\begin{cases}
(a+l+2k+ck)v\otimes t^{a+l+2k}, & I=\emptyset,\\
D_{t\xi_i}v\otimes t^{a+l+2k+1-\epsilon_i}, & I=\{i\},\\
0, & |I|>2,\\
(E_{sr}-E_{rs})v\otimes t^{a+l+2k+2-\epsilon_s-\epsilon_r}, & I=\{r,s\}, r<s;
\end{cases}\\
D_{t^{k+1}\xi_I}.\mathbf{D}_J\bar{v}\otimes t^{a+l}&=[D_{t^{k+1}\xi_I},D_{t\xi_{j_1}}]\mathbf{D}_{J\setminus \{j_1\}}\overline{v}\otimes t^{a+l-(1-\epsilon_{j_1})}\nonumber\\
&+(-1)^{|I|}D_{t\xi_{j_1}}D_{t^{k+1}\xi_I}\mathbf{D}_{J\setminus\{j_1\}}\overline{v}\otimes t^{a+l-(1-\epsilon_{j_1})}.
\end{align}
So, we have
\begin{theorem}
Any simple module in $\sJ_a(\cK(N;\epsilon),\cA)$ is isomorphic to some simple quotient of some $\cF_a(V,c)$ for some simple finite dimensional $\bC z\oplus\fs\fo_N$ module $V$ with $z$ acting as $c\in\bC$.
\end{theorem}
\end{comment}

\section{Main results}
\label{Mainresults}

In this section, we will determine all simple objects in $\sF(\widehat{\cK}(N;\epsilon))$. First of all, from the representation theory of Virasoro algebra, we know that, for $N\le 3$, $C$ acts trivially on any simple module in $\sB(\widehat{\cK}(N;\epsilon))$, and hence $\sB(\cK(N;\epsilon))$ is naturally equivalent to $\sB(\widehat{\cK}(N;\epsilon))$.

Lemma \ref{roughclassification} gives a rough classification of simple quasifinite modules over $\widehat{\cK}(N;\epsilon)$.
\begin{lemma}\label{roughclassification}
Any simple module in $\sF(\widehat{\cK}(N;\epsilon))\setminus\sB(\widehat{\cK}(N;\epsilon))$ is a highest/lowest weight module.
\end{lemma}
This lemma was given in \cite{MZe1} (Proposition 1) without proof. Here, for self consistency, we provide a proof. To prove Lemma \ref{roughclassification}, we need the following well-known result for Virasoro algebra.

\begin{lemma}\label{upperbounded}
Let $M\in\sF(\Vir)$ with $\Supp(M)\subseteq \lambda+2\bZ$. If for any $v\in M$, there exists $N(v)\in\bN$ such that $D_{t^{i+1}}v=0$ for any $i\geq N(v)$, then $\Supp(M)$ is upper bounded.
\end{lemma}

\begin{proof}[Proof of Lemma \ref{roughclassification}]
Let $M\in\sF(\widehat{\cK}(N;\epsilon))\setminus\sB(\widehat{\cK}(N;\epsilon))$ be simple and $\lambda\in\Supp(M)$. We claim that there exists $k\in\bZ$ and a nonzero element $w$ in $M_{\lambda-2k}$ or $M_{\lambda-2k-1}$ such that $D_{t^{k+1}\xi_I}w=D_{t^{k+2}\xi_I}w=0$ for all $I\subseteq\overline{1,N}$. In fact, since $M$ is not bounded, there exists $k$ such that
    \[
    \dim M_{-2k+\lambda}>2^N\sum\limits_{i=0}^{N+1}\dim M_{\lambda+i}
    \]
    or
    \[
    \dim M_{-2k-1+\lambda}>2^N\sum\limits_{i=0}^{N+1}\dim M_{\lambda-1+i},
    \]
    and therefore claim follows. Thus, $D_{t^{i+1}\xi_I}w=0$ for all $i>(k+1)^2$ and $I\subseteq\overline{1,N}$.

    Without loss of generality, we may assume $k\in\bN$. Since $M'=\{v\in M\,|\,D_{t^{i+1}\xi_I}v=0, \forall I\subseteq\overline{1,N}, i\gg0\}$ is a nonzero submodule of $M$, we know $M=M'$. And hence by Lemma \ref{upperbounded}, $\Supp(M)$ is upper bounded, that is $M$ is a highest weight module.
\end{proof}

Let $M\in\sB(\cK(N;\epsilon))$. Consider $\cK(N;\epsilon)$ as the adjoint $\cK(N;\epsilon)$ module. Then the tensor product $\cK(N;\epsilon)\otimes M$ becomes an $\cA\cK(N;\epsilon)$ module under
\[
x\cdot(y\otimes u)=(xy)\otimes u, \forall x\in\cA, y\in\cK(N;\epsilon), u\in M.
\]
Denote $\mathfrak{K}(M)=\{\sum\limits_{i}x_i\otimes v_i\in\cK(N;\epsilon)\otimes M\,|\,\sum\limits_i(ax_i)v_i=0,\forall a\in\cA\}$. Then it is easy to see that $\mathfrak{K}(M)$ is an $\cA\cK(N;\epsilon)$ submodule of $\cK(N;\epsilon)\otimes M$. And hence we have the $\cA\cK(N;\epsilon)$ module $\widehat{M}=(\cK(N;\epsilon)\otimes M)/\mathfrak{K}(M)$. Also, we have a $\cK(N;\epsilon)$ module epimorphism defined by
\[
\pi: \widehat{M}\to\cK(N;\epsilon) M; x\otimes y+\mathfrak{K}(M)\mapsto xy, \forall x\in\cK(N;\epsilon), y\in M.
\]
$\widehat{M}$ is called the $\cA$-cover of $M$ if $\cK(N;\epsilon) M=M$.

\begin{lemma}\label{cover}
Let $M\in\sB(K(N;\epsilon))$ be simple. Then $\widehat{M}\in\sJ(\cK(N;\epsilon),\cA)$.
\end{lemma}
\begin{proof}
It is obvious if $M$ is trivial. Now suppose $M$ is nontrivial. Then $\cK(N;\epsilon)M=M$ and $\Supp(M)\subseteq a+(1+\delta_{\epsilon,\one})\bZ$. Suppose $\dim M_\mu\le r, \forall\mu\in\Supp(M)$. By Lemma \ref{omegaoper}, there exists $m\in\bN$, such that $\Omega_{m,k,p,I}v=0$ for all $k,p\in\bZ, I\subseteq\overline{1,N}$ and $v\in M$. Hence,
\begin{equation}\label{kerner}
\sum\limits_{i=0}^m(-1)^i\binom{m}{i}D_{k-i,I}\otimes D_{p+i,\emptyset}v\in\mathfrak{K}(M).
\end{equation}
Let $S=\mathrm{span}\{D_{k,I}\,|\,0\le k\le m-1, I\subseteq\overline{1,N}\}$. Then $\dim S=2^Nm$ and $S\otimes M$ is a $\bC D_t$ submodule of $\cK(N;\epsilon)\otimes M$ with $\dim(S\otimes M)_\mu\le 2^Nmr, \forall\mu\in a+(1+\delta_{\epsilon,\one})\bZ$. We will prove that $\cK(N;\epsilon)\otimes M=\cK(N;\epsilon)\otimes M_0+S\otimes M+\mathfrak{K}(M)$, from which lemma follows. Indeed, we will prove by induction that for all $u\in M_\mu$ with $\mu\ne0$, $D_{n,I}\otimes u\in S\otimes M+\mathfrak{K}(M)$. We only give proof for $n\ge m$, the proof for $n<0$ is similar. Since $D_t$ acts on $M_\mu$ as a nonzero scalar, we can write $u=D_{1,\emptyset}v$ for some $v\in M_\mu$. Then by (\ref{kerner}) and induction hypothesis, we have
\begin{align*}
D_{n,I}\otimes D_{1,\emptyset}v&=\sum\limits_{i=0}^m(-1)^i\binom{m}{i}D_{n-i,I}\otimes D_{1+i,\emptyset}v-\sum\limits_{i=1}^m(-1)^i\binom{m}{i}D_{n-i,I}\otimes D_{1+i,\emptyset}v\\
&\in S\otimes M+\mathfrak{K}(M).\qedhere
\end{align*}
\end{proof}

Now we can classify all simple module in $\sB(\cK(N;\epsilon))$ for $N\ne4$.
\begin{theorem}\label{bounded}
For $N\neq 4$, any nontrivial simple module in $\sB(\cK(N;\epsilon))$ is a simple quotient of $\cF_a(V,c)$ for some $a,c\in\bC$ and some simple finite dimensional $\fs\fo_N$ module $V$.
\end{theorem}
\begin{proof}
Let $M$ be any nontrivial module in $\sB(\cK(N;\epsilon))$. Then $\cK(N;\epsilon)M=M$ and there is an epimorphism $\pi: \widehat{M}\to M$. From Lemma \ref{cover}, $\widehat{M}$ is bounded. Hence, $\widehat{M}$ has a composition series of $\cA\cK(N;\epsilon)$ submodules
\[
0=\widehat{M}^{(0)}\subset\widehat{M}^{(1)}\subset\cdots\subset\widehat{M}^{(s)}=\widehat{M}
\]
with $\widehat{M}^{(i)}/\widehat{M}^{(i-1)}$ being simple $\cA\cK(N;\epsilon)$ modules. Let $k$ be the minimal integer such that $\pi(\widehat{M}^{(k)})\ne0$. Then we have $\pi(\widehat{M}^{(k)})=M,\pi(\widehat{M}^{(k-1)})=0$ since $M$ is simple. So, we have a $\cK(N;\epsilon)$-epimorphism from the simple $\cA\cK(N;\epsilon)$ module $\widehat{M}^{(k)}/\widehat{M}^{(k-1)}$ to $M$. Thus, theorem follows from
\end{proof}

Combining Lemma \ref{roughclassification} and Theorem \ref{bounded}, we ge the following result.
\begin{theorem}\label{mainresult}
For $N\neq 4$, any simple module in $\sF(\widehat{K}(N;\epsilon))$ is a highest weight module, lowest weight module, or a simple quotient of $\cF_a(V,c)$ for some $a,c\in\bC$ and some simple finite dimensional $\fs\fo_N$ module $V$.
\end{theorem}

\begin{remark}
\begin{enumerate}
\item Theorem \ref{mainresult} shows that the Mart\'inez-Zelmanov's conjecture in \cite{MZe1} holds for $\cK(N;\epsilon)$ ($N\ne4$).
\item When $N=1$, we recover Theorem 4.3 in \cite{CLL}, Theorem 4.7 in \cite{CL} as well as Theorem ? in \cite{S1}. When $N=2,\epsilon=\one$, we recover Theorem 5.2 in \cite{LPX}.
\item Highest/lowest weight modules for $K(1;\epsilon)$ is given in \cite{IK1,IK2}. Highest/lowest weight module $K(2;\epsilon)$ were studied in \cite{ST1,ST2} (untwisted secror) and \cite{IK3} (twisted sector). %However, classification of simple highest weight $K(N;\epsilon)$ modules for $N\ge 3$ is still open.
\end{enumerate}
\end{remark}

\section{Simple bounded modules: $N=2$}\label{N=2}

In this section, we will study simple quotients of $\cF_a(V,c)$ for $\cK(2;\epsilon)$. Together with the results in \cite{ST1,ST2,IK3}, we get a complete classification of simple quasifinite modules for the $N=2$ super Virasoro algebras $\widehat{\cK}(2;\epsilon)$.  Since any simple finite dimensional $\fs\fo_2\oplus\bC z$ module is $1$-dimensional, throughout this section, we assume that $V=\bC v$ with $(E_{21}-E_{12})v=bv, zv=cv$ for some $b,c\in\bC$ and denote $\cF_a(V,c)$ by $\cF_a(b,c)$.  From Lemma \ref{equivcat}, we only need to consider $\epsilon=(1,1)$ and $\epsilon=(1,0)$.

\subsection{Untwisted sector ($\epsilon=(1,1)$)}

From Proposition \ref{isomtensor}, we may assume that $a\in[0,2)$. For $l\in\bZ$, let $v_l=v\otimes t^{a+2l}, v^{\pm}_l=(D_{t\xi_1}\pm\boldsymbol{i}D_{t\xi_2})v\otimes t^{a+2l}, w_l=D_{t\xi_1}D_{t\xi_2}v\otimes t^{a+2l}$. Then $\cF_a(b,c)=\sum\limits_{l\in\bZ}(\bC v_l+\bC v^+_l+\bC v_l^-+\bC w_l)$ with actions
\begin{align*}
&D_{t^{k+1}}v_l=(a+2l+(c+2)k)v_{l+k}, \, (D_{t^{k+1}\xi_1}\pm\boldsymbol{i}D_{t^{k+1}\xi_2})v_l=v_{l+k}^{\pm},\, D_{t^{k+1}\xi_1\xi_2}v_l=bv_{l+k},\\
&D_{t^{k+1}}v_l^{\pm}=(a+2l+(c+1)k)v_{l+k}^{\pm},\, D_{t^{k+1}\xi_1\xi_2}v_l^{\pm}=(b\mp\boldsymbol{i})v_{l+k}^{\pm},\\
&(D_{t^{k+1}\xi_1}+\boldsymbol{i}D_{t^{k+1}\xi_2})v_l^+=0, \, (D_{t^{k+1}\xi_1}-\boldsymbol{i}D_{t^{k+1}\xi_2})v_l^-=0,\\ &(D_{t^{k+1}\xi_1}-\boldsymbol{i}D_{t^{k+1}\xi_2})v_l^+=-(a+2l+2(c+b\boldsymbol{i}+1)k)v_{l+k}+2\boldsymbol{i}w_{l+k},\\
&(D_{t^{k+1}\xi_1}+\boldsymbol{i}D_{t^{k+1}\xi_2})v_l^-=-(a+2l+2(c-b\boldsymbol{i}+1)k)v_{l+k}-2\boldsymbol{i}w_{l+k},\\
&D_{t^{k+1}}w_l=bk^2v_{l+k}+(a+2l+ck)w_{l+k},\, D_{t^{k+1}\xi_1\xi_2}=-ckv_{l+k}+bw_{l+k},\\
&(D_{t^{k+1}\xi_1}\pm\boldsymbol{i}d_{t^{k+1}\xi_2})w_l=(bk\pm\frac{\boldsymbol{i}}{2}(a+2l+2ck))v_{l+k}^{\pm}.
\end{align*}
$\cF_a(b,c)$ is isomorphic to the module $\cR_{\frac{a}{2},\frac{c}{2}+1,b\boldsymbol{i}}$ defined in \cite{LPX} via $v_l\mapsto v_l, \frac{1}{\sqrt{2}}v_l^\pm\mapsto v_l^\pm, -\frac{1}{2}(a+2l)v_l-\boldsymbol{i}w_l\mapsto v_l^{+-}$. Therefore, we have
\begin{proposition}
$\cF_a(b,c)$ is simple if and only if $b^2+c^2\ne0$.
\end{proposition}

\begin{comment}
\begin{proposition}
If $c\pm\boldsymbol{i}b\ne0$, then $\cF_a(b,c)$ is simple.
\end{proposition}
\begin{proof}
Let $W$ be a nonzero submodule of $\cF_a(b,c)$. Then since $\cF_a(b,c)$ is a weight $\bC D_t+\bC D_{t\xi_1\xi_2}$ module, there exists some $l\in\bZ$, such that $\bC v_l+\bC w_l\cap W\ne\{0\}$ or $v_l^+\in W$ or $v_l^-\in W$. If $v_l^+\in W$, then applying $D_{t\xi_1}-\boldsymbol{i}D_{t\xi_2}$, we get $\bC v_l+\bC w_l\cap W\ne\{0\}$. Similarly, if $v_l^-\in W$, one can deduce that $\bC v_l+\bC w_l\cap W\ne\{0\}$. So, we may assume that $0\ne \alpha v_l+\beta w_l\in \bC v_l+\bC w_l\cap W$. Moreover, we may require $\beta=1$. Since
\begin{align*}
D_{t^{k+1}}(\alpha v_l+w_l)&=(\alpha(a+2l+(c+2)k)+bk^2)v_{l+k}+(a+2l+ck)w_{l+k},\\
D_{t^{k+1}\xi_1\xi_2}(\alpha v_l+w_l)&=(\alpha b-ck)v_{l+k}+bw_{l+k},
\end{align*}
and $\begin{vmatrix}
\alpha(a+2l+(c+2)k)+bk^2 & a+2l+ck\\
\alpha b-ck & b
\end{vmatrix}=(b^2+c^2)k^3+(2\alpha b+(a+2l)c)k$. So, there exists $j\in\bZ$ such that $v_j, w_j\in W$. Therefore, $v_k^\pm\in W$ for all $k\in\bZ$. Hence, $w_k\in W$ for all $k\in\bZ$. Finally, from $(b,c)\ne(0,0)$, we get $v_k\in W$ for all $k$ by applying $D_{t^{k+1}}$ (when $b\ne0$) or $D_{t^{k+1}\xi_1\xi_2}$ (when $c\ne0$). Thus, $W=\cF_a(b,c)$, which means $\cF_a(b,c)$ is simple.
\end{proof}
\end{comment}

Now let us consider $\cF_a(b,c)$ with $b^2+c^2=0$. In this case, let
\[
W_{a,c}^{\pm}:=\sum\limits_{l\in\bZ}\Big(\bC v_l^{\pm}+\bC(\frac{1}{2}(a+2l)v_l\mp\boldsymbol{i}w_l)\Big).
\]
One can easily check that when $c+\boldsymbol{i}b=0$ ($c-\boldsymbol{i}b=0$, respectively), $W^+_{a,c}$ ($W^-_{a,c}$, respectively) is isomorphic to the module $\cR_{\frac{a}{2},\frac{c+1}{2}}$ defined in \cite{LPX}.
\begin{proposition}\label{submodnonzero1}
Suppose $bc\ne0$ and $b^2+c^2=0$.
\begin{enumerate}
\item If $a\ne0$ or $c\ne-1$, then $W_{a,c}^+$ is the unique maximal submodule of $\cF_a(b,c)$ and $\cF_a(b,c)/W_{a,c}^+\cong\Pi(W_{a,c+1}^+)$.
\item If $a=0$ and $c=-1$, then $W^{\mathrm{sgn}(\frac{1}{\boldsymbol{i}b})}_{0,-1}+\bC v_0^{\mathrm{sgn}(\frac{\boldsymbol{i}}{b})}$ is the unique maximal submodule of $\cF_0(b,-1)$.
\end{enumerate}
\end{proposition}
\begin{proof}
We only prove statements when $c+b\boldsymbol{i}=0$, the arguments when $c-b\boldsymbol{i}=0$ are similar.

First one can easily check that $W_{a,c}^+$ is a submodule and $\cF_a(c\boldsymbol{i},c)/W_{a,c}^+\cong\Pi(W_{a,c+1}^+)$ via $\bar{v}_l\mapsto v_l^+, \bar{v}_l^-\mapsto-2(a+2l)v_l+2\boldsymbol{i}w_l$. Since $W_{a,c+1}^+$ is simple if and only if $a\ne0$ or $c\ne-1$, we know that $W_{a,c}^+$ is a maximal submodule when $a\ne0$ or $c\ne-1$.

Let $W$ be a maximal submodule of $\cF_a(c\boldsymbol{i},c)$. Then since $\cF_a(c\boldsymbol{i},c)$ is a weight $\bC D_{t}+\bC D_{t\xi_1\xi_2}$ module, we know that there exists $l\in\bZ$ such that $W\cap(\bC v_l+\bC w_l)\ne0$ or $v_l^+\in W$ or $v_l^-\in W$.  Moreover, since
\begin{align}
&(D_{t^{k+1}\xi_1}+\boldsymbol{i}D_{t^{k+1}\xi_2})(D_{t\xi_1}+\boldsymbol{i}D_{t\xi_2})v_l^-=-2ckv_{l+k}^+,\label{eq1}\\
&(D_{t^{k+1}\xi_1}+\boldsymbol{i}D_{t^{k+1}\xi_2})(\alpha v_l+\beta w_l)=\Big(\alpha+\boldsymbol{i}\beta(\frac{a}{2}+l+2ck)\Big)v_{l+k}^+,\nonumber
\end{align}
We may always assume that $v_l^+\in W$ for some $l$.

If $a\ne0$ or $c\ne1$, then $W_{a,c}^+$ is simple. Hence, when $a\ne0$ or $c\ne1$, then $W_{a,c}^+\subseteq W$. So when $a\ne0$ or $c\ne\pm1$, $W=W_{a,c}^+$.

When $a=0,c=1$, then $\sum\limits_{k\ne0}\bC v_k^++\sum\limits_{j\in\bZ}\bC(jv_j-\boldsymbol{i}w_j)\subseteq W$. If $W\ne W_{a,c}^+$, then there exists $l\in\bZ$ such that $v_l^-\in W$ or $v_l\in W$. Moreover, applying $D_{t^{k+1}}$, we may assume $l\ne0$. If $v_l\in W$, then $v_0^+\in W$, a contradiction. If $v_l^-\in W$ with $l\ne0$, then (\ref{eq1}) tells us that $v_0^+\in W$, which is a contradiction. Thus, when $a=0, c\ne-1$, $W_{a,c}^+$ is the unique maximal submodule.

Finally, suppose $a=0,c=-1$. One can easily check $W_{a,c}^++\bC v_0^-$ is a maximal submodule. If $W\ne W_{a,c}^++\bC v_0^-$, then there exits $l$ such that $v_l\in W$ or $v_l^-\in W$. However, from the action of $D_{t^{k+1}\xi_1}-\boldsymbol{i}D_{t^{k+1}\xi_2}$ ($D_{t^{k+1}}$, respectively), $v_l\in W$ ($v_l^-\in W$, respectively) would imply $v_0^-\in W$, which is a contradiction. This completes the proof.
\end{proof}

\begin{proposition}
Suppose $b=c=0$.
\begin{enumerate}
\item If $a=0$, then $\{0\}, \bC w_0, W^+,W^-, W^++ W^-$ are all proper submodules of $\cF_0(V,0)$. In particular, $\cF_0(V,0)$ has a unique simple quotient, which is a trivial module.
\item If $a\ne0$, then $W^\pm$ are simple submodules and $\cF_a(V,0)=W^+\oplus W^-$.
\end{enumerate}
\end{proposition}
\begin{proof}
\begin{enumerate}
\item  It is easy to check that when $a=b=c=0$, $\bC w_0, W^\pm$ are submodules. Now let $W$ be a proper submodule of $\cF_0(V,0)$ and $W\ne\{0\}, \bC w_0$. Since $\cF_0(V,0)$ is a weight $\bC D_t+\bC D_{t\xi_1\xi_2}$ module, we know that there exists $l\in\bZ$ such that $W\cap(\bC v_l+\bC w_l)\ne0$ or $v_l^+\in W$ or $v_l^-\in W$.

If $v_l^+\in W$, then from
\[
(D^{t^{k+1}\xi_1}-\boldsymbol{i}D_{t^{k+1}\xi_2})v_l^+=-2(l+k)v_{l+k}+2\boldsymbol{i}w_{l+k},
\]
we know that $-kv_k+\boldsymbol{i}w_k\in W$ for all $k\in\bZ$. Hence, from
\[
(D_{t^{j+1}\xi_1}+\boldsymbol{i}D_{t^{j+1}\xi_2})(-kv_k+\boldsymbol{i}w_k)=-2kv_{k+j}^+,
\]
we know that $v_k^+\in W$ for all $k\in\bZ$. Thus, we have $W^+\subseteq W$. Similarly, if $v_l^-\in W$, then $W^-\subseteq W$.

If $W^+\subsetneq W$, then there exists $k\in\bZ$ such that $v_k\in W$ or $v_k^-\in W$. If $v_0\in W$, then $W=\cF_0(V,0)$, which contradicts with the fact $W$ is proper. If $v_k\in W$ for some $k\ne0$, then
\[
v_k^-=(D_{t\xi_1}-\boldsymbol{i}D_{t\xi_2})v_k\in W.
\]
So we always have $v_k^-\in W$. Thus $W^-\subseteq W$ and hence $W=W^+\oplus W^-$ since $W^++ W^-$ has codimension $1$. Similarly, if $W^-\subsetneq W$, then $W=W^++ W^-$.

Now suppose $\alpha v_l+\beta w_l\in W\setminus\bC w_0$. If $l=0$, then applying $D_{t\xi_1}\pm\boldsymbol{i}D_{t\xi_2}$, we know that $v_l^{\pm}\in W$. Then $W^+\oplus W^-\subseteq W$ and hence $W=\cF_0(V,0)$, contradiction. So $l\ne0$. Also, we may assume $\alpha\pm\boldsymbol{i}l\beta$. Again, applying $D_{t\xi_1}\pm\boldsymbol{i}D_{t\xi_2}$, we have $v_l^+\in W$ or $v_l^-\in W$. And therefore, $W=W^+,W^-$ or $W^++ W^-$.
\item Now suppose $a\ne0$. The fact $\cF_a(V,0)=W^+\oplus W^-$ is obvious. It remains to show $W^\pm$ are simple. Similar as 1, any nonzero submodule of $W^+$ contains some $v_l^+$ or $\frac{1}{2}(a+2l)v_l-\boldsymbol{i}w_l$. Then from the actions of $D_{t^{k+1}\xi_1}+\boldsymbol{i}D_{t^{k+1}\xi_2}$, we get $W^+$ is simple. Similarly, one can show that $W^-$ is simple.\qedhere
\end{enumerate}
\end{proof}

\subsection{Twisted sector ($\epsilon=(1,0)$)}

Now suppose $\epsilon=(1,0)$. First, we have
\begin{proposition}
\begin{enumerate}
\item $\cF_a(b,c)$ is a simple $\cA\cK(2;(1,0))$ module if and only if $b\ne0$.
\item For $\delta\in\{0,1\}$, let $\cF_a(0,c)_\delta:=v\otimes t^{a+\delta}\bC[t^2]+D_{t\xi_1}v\otimes t^{a+\delta}\bC[t^2]+D_{t\xi_2}v\otimes t^{a+\delta+1}\bC[t^2]+D_{t\xi_1}D_{t\xi_2}v\otimes t^{a+\delta+1}\bC[t^2]$. Then $\cF_a(0,c)_\delta$ are simple $\cA\cK(2;(1,0))$ modules and $\cF_a(0,c)=\cF_a(0,c)_0\oplus\cF_a(0,c)_1$.
\item $\cF_a(0,c)_1\cong\cF_{a+1}(0,c)_0$.
\end{enumerate}
\end{proposition}
\begin{proof}
From the actions of $t^k\xi_I$ and $D_t$, we know that any nonzero $\cA\cK(2;(1,0))$ submodule $W$ of $\cF_a(b,c)$ contains $v\otimes t^{a+l}$ for some $l$. Therefore $v\otimes t^{a+l}\bC[t^2]\subseteq W$.

If $b\ne0$, then from $D_{t\xi_1\xi_2}(v\otimes t^{a+l})=bv\otimes t^{a+l+1}$ we know that $v\otimes t^a\bC[t]\subseteq W$. Applying $D_{t\xi_1}$ and $D_{t\xi_2}$ we get $W=\cF_a(b,c)$.

Now suppose $b=0$. Similar argument show that $\cF_a(0,c)_\delta$ are simple $\cA\cK(2;(1,0))$ modules. The statement  $\cF_a(0,c)=\cF_a(0,c)_0\oplus\cF_a(0,c)_1$ is clear. To finish the proof, it remains to show $\cF_a(0,c)_1\cong\cF_{a+1}(0,c)_0$. It is easy to check the map from $\cF_a(0,c)_1$ to $\cF_{a+1}(0,c)_0$ defined by mapping $\mathbf{D}_Iv\otimes t^{a+1+\sum\limits_{i\in I}(1-\epsilon_i)+2k}$ to $\mathbf{D}_Iv\otimes t^{a+1+\sum\limits_{i\in I}(1-\epsilon_i)+2k}$ is an isomorphism.
\end{proof}

Similar to Proposition \ref{isomtensor}, we have
\begin{proposition}
\begin{enumerate}
\item For $a,a'\in\bZ, b,b'\in\bC^*, c,c'\in\bC$, $\cF_a(b,c)\cong\cF_{a'}(b',c')$ as $\cA\cK(2;(1,0))$ modules if and only if $a-a'\in\bZ$, $b=b'$ and $c=c'$.
\item For $a,a'\in\bZ,c,c'\in\bC$, $\cF_a(0,c)_0\cong\cF_{a'}(0,c')_0$ as $\cA\cK(2;(1,0))$ modules if and only if $a-a'\in2\bZ$ and $c=c'$.
\end{enumerate}
\end{proposition}

So, we only need to consider $\cF_a(b,c)$ ($0\le a<1,b\ne0$) and $\cF_a(0,c)_0$ ($0\le a<2$). For $\cF_a(b,c)$, let $v_l=v\otimes t^{a+l}, v_l^{\pm}=(D_{t\xi_1}\pm\boldsymbol{i}D_{t\xi_2})v\otimes t^{a+l}, w_l=D_{t\xi_1}D_{t\xi_2}v\otimes t^{a+l}$. Then $\cF_a(b,c)=\sum\limits_{l\in\bZ}(\bC v_l+\bC v_l^++\bC v_l^-+\bC w_l)$ with actions
\begin{align*}
&D_{t^{k+1}}v_l=(a+l+(c+2)k)v_{l+2k},\,D_{t^{k+1}\xi_1}v_l=\frac{1}{2}(v_{l+2k}^++v_{l+2k}^-),\\
&D_{t^{k+1}\xi_2}v_l=\frac{\boldsymbol{i}}{2}(v_{l+2k+1}^--v_{l+2k+1}^+),\,D_{t^{k+1}\xi_1\xi_2}v_l=bv_{l+2k+1},\\
&D_{t^{k+1}}v_l^{\pm}=(a+l+(c+1)k)v_{l+2k}^{\pm},\, D_{t^{k+1}\xi_1\xi_2}v_l^{\pm}=(b\mp\boldsymbol{i})v_{l+2k+1}^{\pm},\\
&D_{t^{k+1}\xi_1}v_l^{\pm}=-\frac{1}{2}(a+l+2(c\pm b\boldsymbol{i}+1)k)v_{l+2k}\pm\boldsymbol{i}w_{l+2k},\\
&D_{t^{k+1}\xi_2}v_l^{\pm}=\mp\frac{\boldsymbol{i}}{2}(a+l+2(c\pm b\boldsymbol{i}+1)k)v_{l+2k+1}-w_{l+2k+1},\\
&D_{t^{k+1}}w_l=k^2bv_{l+2k}+(a+l+ck)w_{l+2k},\\
&D_{t^{k+1}\xi_1\xi_2}w_l=-\frac{1}{2}(2k+1)cv_{l+2k+1}+bw_{l+2k+1},\\ &D_{t^{k+1}\xi_1}w_l=\frac{\boldsymbol{i}}{4}(a+l+2(c-b\boldsymbol{i})k)v_{l+2k}^+-\frac{\boldsymbol{i}}{4}(a+l+2(c+b\boldsymbol{i})k)v_{l+2k}^-,\\
&D_{t^{k+1}\xi_2}w_l=\frac{1}{4}(a+l+(c-b\boldsymbol{i})(2k+1))v_{l+2k+1}^++\frac{1}{4}(a+l+(c+b\boldsymbol{i})(2k+1))v_{l+2k+1}^-.
\end{align*}
\begin{proposition}
If $b(b^2+c^2)\ne0$, then $\cF_a(b,c)$ is a simple $\cK(2;(1,0))$ module.
\end{proposition}
\begin{proof}
Let $W$ be a nonzero submodule of $\cF_a(b,c)$. From
\begin{align}
&D_tv_l=(a+l)v_l, D_tv_l^\pm=(a+l)v_l^\pm,D_tw_l=w_l,\label{eq21}\\
&(D_{\xi_1\xi_2}D_{t\xi_1\xi_2})v_l=b^2v_l, (D_{\xi_1\xi_2}D_{t\xi_1\xi_2})v_l^{\pm}=(b\mp\boldsymbol{i})^2v_l^\pm,(D_{\xi_1\xi_2}D_{t\xi_1\xi_2})w_l=b^2w_l,\label{eq22}
\end{align}
we know that $W$ contains $v_l^+,v_l^-$ or $\alpha v_l+\beta w_l$ for some $l$. We claim that either $v_l^+\in W$ for some $l$ or $v_l^-\in W$ for some $l$. Indeed, if $0\ne\alpha v_l+\beta w_l\in W$, then from
\begin{align*}
&(D_{\xi_1\xi_2}-(b-\boldsymbol{i})^2)D_{t\xi_1}(\alpha v_l+\beta w_l)=(2\alpha\boldsymbol{i}+\beta(a+l))bv_l^-,\\
&(D_{\xi_1\xi_2}-(b+\boldsymbol{i})^2)D_{t\xi_1}(\alpha v_l+\beta w_l)=-(2\alpha\boldsymbol{i}-\beta(a+l))bv_l^+,
\end{align*}
we know that claim holds when $a\ne0$ or $\alpha\ne0$. If $a=0$ and $\alpha=0$, then $\beta\ne0$ and claim holds unless $l=0$, that is $w_0\in W$. However, $D_1w_0=bv_{l-2}-cw_{l-2}\in W$ implies that $v_{l-2}^+\in W$ or $v_{l-2}^-\in W$. Thus, claim holds. Without loss of generality, we may assume $v_l^+\in W$. Applying $D_{t^{k+1}\xi_1\xi_2}$, we have $v_{l+2k+1}^+\in W$ for all $k\in\bZ$. Therefore
$W\ni D_{t^k\xi_2}v_{l+1}^+-D_{t^{k+1}\xi_2}v_{l-1}^+=\boldsymbol{i}(c+b\boldsymbol{i})v_{l+2k}$ for all $k$. And hence $bv_{l+2k+1}=D_{t\xi_1\xi_2}v_{l+2k}\in W$. Then, from the actions of $D_{t\xi_1}$ and $D_{t\xi_2}$, we know $W=\cF_a(b,c)$, that is $\cF_a(b,c)$ is simple.
\end{proof}

Now let us consider $\cF_a(b,c)$ with $b\ne0$ and $b^2+c^2=0$. In this case, let $W_{a,c}^\pm:=\sum\limits_{l\in\bZ}\left(\bC v_l^\pm+\bC((a+l)v_l\mp2\boldsymbol{i}w_l)\right)$. Similar to Proposition \ref{submodnonzero1}, we get

\begin{proposition}
Suppose $b\ne0$ and $b^2+c^2=0$.
\begin{enumerate}
\item If $a\ne0$ or $c\ne-1$, then $W_{a,c}^{\mathrm{sgn}(\frac{b}{c\boldsymbol{i}})}$ is the unique maximal submodule of $\cF_a(b,c)$ and $\cF_a(b,c)/W_{a,c}^{\mathrm{sgn}(\frac{b}{c\boldsymbol{i}})}\cong\Pi(W_{a,c+1}^{\mathrm{sgn}(\frac{b}{c\boldsymbol{i}})})$.
\item If $a=0$ and $c=-1$, then $W^{\mathrm{sgn}(\frac{1}{\boldsymbol{i}b})}_{0,-1}+\bC v_0^{\mathrm{sgn}(\frac{\boldsymbol{i}}{b})}$ is the unique maximal submodule of $\cF_0(b,-1)$.
\end{enumerate}
\end{proposition}

To summarize, we have
\begin{theorem}
$\cF_a(b,c)$ is a simple $\cK(2;(1,0))$ module if and only if $b(b^2+c^2)\ne0$.
\end{theorem}

Now let us consider $\cF_a(0,c)_0$. The following result can be easily proved.
\begin{theorem}
$\cF_a(0,c)_0$ is simple if and only if $a\ne1$ or $c\ne0$. Moreover, $\cF_1(0,0)_0$ has a unique nonzero proper submodule $\bC w_{-1}$.
\end{theorem}
\begin{proof}
For any nonzero submodule $W$ of $\cF_a(0,c)_0$, from (\ref{eq21}) and (\ref{eq22}), we know that there exists some $l\in\bZ$, such that $v_{2l}\in W$ or $D_{t\xi_1}v_{2l}\in W$ or $D_{t\xi_2}v_{2l+1}\in W$ or $w_{2l+1}\in W$. Hence, we know that $w_{2l+1}\in W$ for some $l$.

If $c\ne0$, then applying $D_{t^{k+1}\xi_1\xi_2}$, we get $v_{2j}\in W$ for all $j\in\bZ$. And therefore $W=\cF_a(0,c)_0$.

The fact that $\cF_a(0,0)_0$ is simple when $a\ne1$ follows from
\begin{align*}
&D_{t^{k+1}}w_{2l+1}=(a+2l+1)w_{2(l+k)+1},\\
&D_{t^{k+1}\xi_1}w_{2j+1}=-\frac{1}{2}(a+2j+1)D_{t\xi_2}v\otimes t^{a+2(j+k)+1},\\
&D_{t^{k+1}\xi_2}w_{2j+1}=\frac{1}{2}(a+2j+1)D_{t\xi_1}v\otimes t^{a+2(j+k)+2},\\
&D_{t\xi_2}.D_{t\xi_2}v\otimes t^{a+2j+1}=-\frac{1}{2}(a+2j+1)v_{2(j+1)}.
\end{align*}

Now suppose $a=1,c=0$. In this case, it is obvious that $\bC w_{-1}$ is a nonzero submodule of $\cF_1(0,0)_0$. If $w_l\in W$ with $l\ne-1$, then from above equations, we get $w_{2j+1}, D_{t\xi_1}v\otimes t^{2j}, D_{t\xi_2}v\otimes t^{2j+1}\in W$ for all $j$. Then $W=\cF_1(0,0)_0$ since $D_{t\xi_1}D_{t\xi_1}v\otimes t^{2j}=-\frac{1}{2}(2j+1)v_{2j}$.
\end{proof}

\section{Simple bounded modules: $N=3$}\label{N=3}

We ends this paper by studying simple quotients of $\cF_a(V,c)$ for $\cK(3;\epsilon)$. From Lemma \ref{equivcat}, we only need to consider $\epsilon=(1,1,1)$ and $\epsilon=(0,0,0)$. Recall from \cite{Hu} that any finite dimensional simple $\fs\fo_3$ module is isomorphic to some $V_m:=\sum\limits_{j=0}^{2m}\bC v_j$ with $m\in\frac{1}{2}\bZ_+$ and actions
\begin{align*}
&2\boldsymbol{i}(E_{21}-E_{12})v_j=2(m-j)v_j,\\
&(E_{31}-E_{13}+\boldsymbol{i}(E_{32}-E_{23}))v_j=jv_{j-1},\\
&(E_{13}-E_{31}+\boldsymbol{i}(E_{32}-E_{23}))v_j=(2m-j)v_{j+1},
\end{align*}
here $0\le j\le 2m+1,v_{-1}=v_{2m+1}=0$. For $\epsilon=(\varepsilon,\varepsilon,\varepsilon)$, set
\[\cF_a(m,c,\varepsilon):=\begin{cases}
\cF_a(V_m,c), & \varepsilon=1,\\
\cF_a(V_m,c)_0, & \varepsilon=0
\end{cases}.\]
For $0\le r\le 2m+1, l\in\bZ,\varepsilon\in\{0,1\}$, let $v_{r,l}=v_r\otimes t^{a+2l}, v_{r,l,\varepsilon}^{\pm}=(D_{t\xi_1}\pm\boldsymbol{i}D_{t\xi_2})v_r\otimes t^{a+2l+1-\varepsilon}, w_{r,l,\varepsilon}=D_{t\xi_3}v_r\otimes t^{a+2l+1-\varepsilon}, w_{r,l}^{\pm}=(D_{t\xi_1}D_{t\xi_3}\pm\boldsymbol{i}D_{t\xi_2}D_{t\xi_3})v_r\otimes t^{a+2l}, u_{r,l}=D_{t\xi_1}D_{t\xi_2}v_r\otimes t^{a+2l}, \bar{u}_{r,l,\varepsilon}=D_{t\xi_1}D_{t\xi_2}D_{t\xi_3}v_r\otimes t^{a+2l+1-\varepsilon}$. Then \[\cF_a(m,c,\varepsilon)=\sum\limits_{l\in\bZ}\sum\limits_{r=0}^{2m+1}(\bC v_{r,l}+\bC v_{r,l,\varepsilon}^++\bC v_{r,l,\varepsilon}^-+\bC w_{r,l,\varepsilon}+\bC w_{r,l}^++\bC w_{r,l}^-+\bC u_{r,l}+\bC \bar{u}_{r,l,\varepsilon})\]
with actions
\begin{align*}
&D_{t^{k+1}}v_{r,l}=(a+2l+(c+2)k)v_{r,l+k}, \, (D_{t^{k+1}\xi_1}\pm\boldsymbol{i}D_{t^{k+1}\xi_2})v_{r,l}=v_{r,l+k,\epsilon}^\pm,\\ &D_{t^{k+1}\xi_3}v_{r,l}=w_{r,l+k,\varepsilon},\, D_{t^{k+1}\xi_1\xi_2}v_{r,l}=-\boldsymbol{i}(m-r)v_{r,l+k+1-\varepsilon},\, D_{t^{k+1}\xi_1\xi_2\xi_3}v_{r,l}=0,\\
&(D_{t^{k+1}\xi_1\xi_3}\pm\boldsymbol{i}D_{t^{k+1}\xi_2\xi_3})v_{r,l}=(r-m\pm m)v_{r\pm1,l+k+1-\varepsilon}, \\ %&(D_{t^{k+1}\xi_1\xi_3}-\boldsymbol{i}D_{t^{k+1}\xi_2\xi_3})v_{r,l}=-(2m-r)v_{r+1,l+k+1-\varepsilon};\\
&D_{t^{k+1}}v_{r,l,\varepsilon}^\pm=(a+2l+1-\varepsilon+(c+1)k)v_{r,l+k,\varepsilon}^\pm,\\
&(D_{t^{k+1}\xi_1}+\boldsymbol{i}D_{t^{k+1}\xi_2})v_{r,l,\varepsilon}^+=0,\,(D_{t^{k+1}\xi_1}-\boldsymbol{i}D_{t^{k+1}\xi_2})v_{r,l,\varepsilon}^-=0,\\
&(D_{t^{k+1}\xi_1}-\boldsymbol{i}D_{t^{k+1}\xi_2})v_{r,l,\varepsilon}^+\\
=&-(a+2l+(c+2)(1-\varepsilon)+2(c+1+m-r)k)v_{r,l+k+1-\varepsilon}+2\boldsymbol{i}u_{r,l+k+1-\varepsilon},\\
&(D_{t^{k+1}\xi_1}+\boldsymbol{i}D_{t^{k+1}\xi_2})v_{r,l,\varepsilon}^-\\
=&-(a+2l+(c+2)(1-\varepsilon)+2(c+1+r-m)k)v_{r,l+k+1-\varepsilon}-2\boldsymbol{i}u_{r,l+k+1-\varepsilon},\\
&D_{t^{k+1}\xi_3}v_{r,l,\varepsilon}^\pm=(r-m\pm m)kv_{r\mp1,l+k+1-\varepsilon}-w_{r,l+k+1-\varepsilon}^\pm,\\
%&D_{t^{k+1}\xi_3}v_{r,l,\varepsilon}^-=-(2m-r)kv_{r-1,l+k+1-\varepsilon}-w_{r,l+k+1-\varepsilon}^-,\\
&D_{t^{k+1}\xi_1\xi_2}v_{r,l,\varepsilon}^\pm=-\boldsymbol{i}(m-r\pm1)v_{r,l+k+1-\varepsilon,\varepsilon}^\pm,\\
&(D_{t^{k+1}\xi_1\xi_3}+\boldsymbol{i}D_{t^{k+1}\xi_2\xi_3})v_{r,l,\varepsilon}^\pm=rv_{r-1,l+k+1-\varepsilon,\varepsilon}^\pm+(1\mp1)w_{r,l+k+1-\varepsilon},\\
&(D_{t^{k+1}\xi_1\xi_3}-\boldsymbol{i}D_{t^{k+1}\xi_2\xi_3})v_{r,l,\varepsilon}^\pm=(r-2m)v_{r+1,l+k+1-\varepsilon,\varepsilon}^\pm+(1\pm1)w_{r,l+k+1-\varepsilon},\\
&D_{t^{k+1}\xi_1\xi_2\xi_3}v_{r,l,\varepsilon}^\pm=\boldsymbol{i}(m\pm(r-m))v_{r\mp1,l+k+2(1-\varepsilon)};\\
&D_{t^{k+1}}w_{r,l,\varepsilon}=(a+2l+1-\varepsilon+(c+1)k)w_{r,l+k,\varepsilon}, \\
&(D_{t^{k+1}\xi_1}\pm\boldsymbol{i}D_{t^{k+1}\xi_2})w_{r,l,\varepsilon}=(m-r\mp m)kv_{r\mp1,l+k+1-\varepsilon}+w_{r,l+k+1-\varepsilon}^\pm,\\
&D_{t^{k+1}\xi_3}w_{r,l,\varepsilon}=-\frac{1}{2}(a+2l+(c+2)(1-\varepsilon)+2(c+1)k)v_{r,l+k+1-\varepsilon},\\
&(D_{t^{k+1}\xi_1\xi_3}\pm\boldsymbol{i}D_{t^{k+1}\xi_2\xi_3})w_{r,l,\varepsilon}=-v_{r,l+k+1-\varepsilon,\varepsilon}^\pm+(r-m\pm m)w_{r\mp1,l+k+1-\varepsilon,\varepsilon},\\
&D_{t^{k+1}\xi_1\xi_2}w_{r,l,\varepsilon}=\boldsymbol{i}(r-m)w_{r,l+k+1-\varepsilon,\varepsilon},,\,D_{^{k+1}\xi_1\xi_2\xi_3}w_{r,l,\varepsilon}=\boldsymbol{i}(m-r)v_{r,l+k+2(1-\varepsilon)};\\
&D_{t^{k+1}}w_{r,l}^\pm=(r-m\pm m)k(k+1-\varepsilon)v_{r-mp1,l+k}+(a+2l+ck)w_{r,l+k}^\pm,\\
&(D_{t^{k+1}\xi_1}+\boldsymbol{i}D_{t^{k+1}\xi_2})w_{r,l}^\pm\\
=&rkv_{r-1,l+k,\varepsilon}^\pm-(1\mp1)\left(\frac{1}{2}(a+2l+c(1-\varepsilon)+2(c+r-m)k)w_{r,l+k,\varepsilon}+\boldsymbol{i}\bar{u}_{r,l+k,\varepsilon}\right),\\
&(D_{t^{k+1}\xi_1}-\boldsymbol{i}D_{t^{k+1}\xi_2})w_{r,l}^\pm\\
=&(r-2m)kv_{r+1,l+k,\varepsilon}^\pm-(1\pm1)\left(\frac{1}{2}(a+2l+c(1-\varepsilon)+2(c+m-r)k)w_{r,l+k,\varepsilon}-\boldsymbol{i}\bar{u}_{r,l+k,\varepsilon}\right),\\
&D_{t^{k+1}\xi_3}w_{r,l}^\pm=\frac{1}{2}(a+2l+2ck+c(1-\varepsilon))v_{r,l+k,\varepsilon}^\pm+(r-m\mp m)kw_{r\mp1,l+k,\varepsilon},\\
&D_{t^{k+1}\xi_1\xi_2}w_{r,l}^\pm=\boldsymbol{i}(m\pm(r-m))(k+1-\varepsilon)v_{r\mp1,l+k+1-\varepsilon}-\boldsymbol{i}(m-r\pm1)w_{r,l+k+1-\varepsilon}^\pm,\\
&(D_{t^{k+1}\xi_1\xi_3}+\boldsymbol{i}D_{t^{k+1}\xi_2\xi_3})w_{r,l}^\pm\\
=&rw_{r-1,l+k+1-\varepsilon}^\pm+(1\mp1)\left((m-r-c)(k+1-\varepsilon)v_{r,l+k+1-\varepsilon}-\boldsymbol{i}u_{r,l+k+1-\varepsilon}\right),\\
&(D_{t^{k+1}\xi_1\xi_3}-\boldsymbol{i}D_{t^{k+1}\xi_2\xi_3})w_{r,l}^\pm\\
=&(r-2m)w_{r+1,l+k+1-\varepsilon}^\pm-(1\pm1)\left((m-r+c)(k+1-\varepsilon)v_{r,l+k+1-\varepsilon}-\boldsymbol{i}u_{r,l+k+1-\varepsilon}\right),\\
&D_{t^{k+1}\xi_1\xi_2\xi_3}w_{r,l}^\pm=-\boldsymbol{i}(m-r\pm1)v_{r,l+k+1-\varepsilon,\varepsilon}^\pm+\boldsymbol{i}(m\pm(r-m))w_{r\mp1,l+k+1-\varepsilon,\varepsilon};\\
&D_{t^{k+1}}u_{r,l}=\boldsymbol{i}(r-m)k(k+1-\varepsilon)v_{r,l+k}+(a+2l+ck)u_{r,l+k},\\
&(D_{t^{k+1}\xi_1}\pm\boldsymbol{i}D_{t^{k+1}\xi_2})u_{r,l}=\frac{\boldsymbol{i}}{2}\left(2(r-m)k\pm(a+2l+c(2k+1-\varepsilon))\right)v_{r,l+k,\varepsilon}^\pm,\\
&D_{t^{k+1}\xi_3}u_{r,l}=\frac{\boldsymbol{i}}{2}rkv_{r-1,l+k,\varepsilon}^-+\frac{\boldsymbol{i}}{2}(2m-r)kv_{r+1,l+k,\varepsilon}^++\bar{u}_{r,l+k,\varepsilon},\\
&D_{t^{k+1}\xi_1\xi_2}u_{r,l}=\boldsymbol{i}(r-m)u_{r,l+k+1-\varepsilon}-c(k+1-\varepsilon)v_{r,l+k+1-\varepsilon},\\
&(D_{t^{k+1}\xi_1\xi_3\pm\boldsymbol{i}}D_{t^{k+1}\xi_2\xi_3})u_{r,l}\\
=&\boldsymbol{i}(-m\mp(r-m))(k+1-\varepsilon)v_{r\mp1,l+k+1-\varepsilon}+(r-m\pm m)u_{r\mp1,l+k+1-\varepsilon}\pm\boldsymbol{i}w_{r,l+k+1-\varepsilon}^\pm,\\
&D_{t^{k+1}\xi_1\xi_2\xi_3}u_{r,l}=-\frac{1}{2}rv_{r-1,l+k+1-\varepsilon,\varepsilon}^-+\frac{1}{2}(2m-r)v_{r+1,l+k+1-\varepsilon,\varepsilon}^+-w_{r,l+k+1-\varepsilon,\varepsilon};\\
&D_{t^{k+1}}\bar{u}_{r,l,\varepsilon}=-\frac{\boldsymbol{i}}{2}rk(k+1-\varepsilon)v_{r-1,l+k,\varepsilon}^--\frac{\boldsymbol{i}}{2}(2m-r)k(k+1-\varepsilon)v_{r+1,l+k,\varepsilon}^+\\
&\hskip2cm-\boldsymbol{i}(m-r)k(k+1-\varepsilon)w_{r,l+k,\varepsilon}+(a+2l+(c-1)k+1-\varepsilon)\bar{u}_{r,l+k,\varepsilon},\\
&(D_{t^{k+1}\xi_1}\pm\boldsymbol{i}D_{t^{k+1}\xi_2})\bar{u}_{r,l,\varepsilon}=\boldsymbol{i}(m\pm(r-m))k(k+1-\varepsilon)v_{r\mp1,l+k+1-\varepsilon}+(m-r\mp m)ku_{r\mp1,l+k+1-\varepsilon}\\
&\hskip2.5cm+\frac{\boldsymbol{i}}{2}\left(2(r-m)k\pm(a+2l+c(2k-1-\varepsilon))\right)w_{r,l+k+1-\varepsilon}^\pm,\\
&D_{t^{k+1}\xi_3}\bar{u}_{r,l,\varepsilon}=\boldsymbol{i}(m-r)k(k+1-\varepsilon)v_{r,l+k+1-\varepsilon}+\frac{\boldsymbol{i}}{2}rkw_{r-1,l+k+1-\varepsilon}^-+\frac{\boldsymbol{i}}{2}(2m-r)kw_{r+1,l+k+1-\varepsilon}^+\\
&\hskip2cm-\frac{1}{2}(a+2l+c(1-\varepsilon)+2(c-1)k)u_{r,l+k+1-\varepsilon},\\
&D_{t^{k+1}\xi_1\xi_2}\bar{u}_{r,l,\varepsilon}=\frac{1}r(k+1-\varepsilon)v_{r-1,l+k+1-\varepsilon,\varepsilon}^--\frac{1}{2}(2m-r)(k+1-\varepsilon)v_{r+1,l+k+1-\varepsilon,\varepsilon}^+\\
&\hskip2cm+(1-c)(k+1-\varepsilon)w_{r,l+k+1-\varepsilon,\varepsilon}w_{r,l+k+1-\varepsilon,\varepsilon}-\boldsymbol{i}(m-r)\bar{u}_{r,l+k+1-\varepsilon,\varepsilon},\\
&(D_{t^{k+1}\xi_1\xi_3}\pm\boldsymbol{i}D_{t^{k+1}\xi_2\xi_3})\bar{u}_{r,l,\varepsilon}=\boldsymbol{i}(m-r\pm(1-c))(k+1-\varepsilon)v_{r,l+k+1-\varepsilon,\varepsilon}^\pm\\
&\hskip2cm-\boldsymbol{i}(m\pm(r-m))(k+1-\varepsilon)w_{\mp1,l+k+1-\varepsilon,\varepsilon}+(r-m\pm m)\bar{u}_{r\mp1,l+k+1-\varepsilon,\varepsilon},\\
&D_{t^{k+1}\xi_1\xi_2\xi_3}\bar{u}_{r,l,\varepsilon}=-\frac{1}{2}(a+2l+2(1-c)k+(4-3c)(1 -\varepsilon))v_{r,l+k+2(1-\varepsilon)}+\boldsymbol{i}(m-r)u_{r,l+k+2(1-\varepsilon)}\\
&\hskip2cm-\frac{1}{2}rw_{r-1,l+k+2(1-\varepsilon)}^-+\frac{1}{2}(2m-r)w_{r+1,l+k+2(1-\varepsilon)}^+.
\end{align*}

First we have
\begin{lemma}\label{condition}
Suppose $c\ne-m$ and $W$ be a nonzero submodule of $\cF_a(V_m,c,\varepsilon)$ such that $W\cap(\bC v_{0,l,\varepsilon}^++\bC w_{0,l}^+)\ne0$ for some $l$, then $W=\cF_a(V_m,c,\varepsilon)$.
\end{lemma}
\begin{proof}
From
$D_{t^{k+1}\xi_1\xi_2\xi_3}v_{0,l,\varepsilon}^+=0, D_{t^{k+1}\xi_1\xi_2\xi_3}w_{0,l}^+=-\boldsymbol{i}v_{0,l+k+1-\varepsilon,\varepsilon}^+$, we know that there exists some $l$ such that $v_{0,l,\varepsilon}^+\in W$. Hence for any $k\in\bZ$, $w_{0,l+k+1-\varepsilon}^+=-D_{t^{k+1}\xi_3}v_{0,l,\varepsilon}^+\in W$. Therefore, for any $j\in\bZ$, we have
\[
-2(m+c)v_{0,j+1-\varepsilon}=(D_{t^{2}\xi_1\xi_3}-\boldsymbol{i}D_{t^{2}\xi_2\xi_3})w_{0,j-1}^+-(D_{t\xi_1\xi_3}-\boldsymbol{i}D_{t\xi_2\xi_3})w_{0,j}^+\in W.
\]
So $W=\cF_a(V_m,c)$.
\end{proof}

\begin{theorem}
$\cF_a(0,c,\varepsilon)$ is simple if and only if $c\ne0,1$ or $a+c(1-\varepsilon)\notin2\bZ$. Moreover,
\begin{enumerate}
\item $\cF_0(0,0,\varepsilon)$ has a unique proper submodule spanned by $\{v_{0,j}, v_{0,l}^\pm, w_{0,l}^\pm, w_{0,l}, u_{0,l}, \bar{u}_{0,l}\}$, denoted by $\cF'_0(0,0,\varepsilon)$. The corresponding quotient is trivial.
\item $\cF_{1-\varepsilon}(0,1,\varepsilon)$ has a unique proper submodule $\bC\bar{u}_{0,1-\varepsilon,\varepsilon}$, which is trivial. And $\cF_{1-\varepsilon}(0,1,\varepsilon)/\bC\bar{u}_{0,0,\varepsilon}\cong\cF'_0(0,0,\varepsilon)$.
\end{enumerate}
\end{theorem}
\begin{proof}
Let $W$ be a nonzero submodule of $\cF_a(0,c,\varepsilon)$.

From the fact that $\cF_a(0,c,1)$ is a weight $\bC D_t+\bC D_{t\xi_1\xi_2}$ module, we have $W\cap(\bC v_{0,l,1}^++\bC w_{0,l}^+)\ne0$ or $W\cap(\bC v_{0,l,1}^-+\bC w_{0,l}^-)\ne0$ or $W\cap\big(\bC v_{0,l}+\bC w_{0,l}+\bC(u_{0,l}-\frac{\boldsymbol{i}}{2}(a+2l)v_{0,l})+\bC(\bar{u}_{0,l,1}-\frac{\boldsymbol{i}}{2}(a+2l)w_{0,l})\big)\ne0$. Then from
\begin{align*}
&(D_{t\xi_1}+\boldsymbol{i}D_{t\xi_2})v_{0,l,1}^-=-2\boldsymbol{i}(u_{0,l}-\frac{\boldsymbol{i}}{2}(a+2l)v_{0,l}),\\
&(D_{t\xi_1}+\boldsymbol{i}D_{t\xi_2})w_{0,l}^-=-(a+2l)w_{0,l,1}-2\boldsymbol{i}\bar{u}_{0,l,1},
\end{align*}
and
\begin{align*}
&(D_{t\xi_1}+\boldsymbol{i}D_{t}\xi_2)v_{0,l}=v_{0,l}^+,\,(D_{t\xi_1}+\boldsymbol{i}D_{t}\xi_2)w_{0,l}=w_{0,l}^+,\\
&(D_{t\xi_1}+\boldsymbol{i}D_{t}\xi_2)(u_{0,l}-\frac{\boldsymbol{i}}{2}(a+2l)v_{0,l})=(D_{t\xi_1}+\boldsymbol{i}D_{t}\xi_2)(\bar{u}_{0,l,1}-\frac{\boldsymbol{i}}{2}(a+2l)w_{0,l})=0,
\end{align*}
we may assume  $W\cap(\bC v_{0,l,1}^++\bC w_{0,l}^+)\ne0$ or $W\cap\big(\bC(u_{0,l}-\frac{\boldsymbol{i}}{2}(a+2l)v_{0,l})+\bC(\bar{u}_{0,l,1}-\frac{\boldsymbol{i}}{2}(a+2l)w_{0,l})\big)\ne0$.

Similarly, when $\varepsilon=0$, one can get that there exists some $l$ such that one of the following holds: $v_{0,l,0}^+\in W, w_{0,l}\in W, u_{0,l}-\frac{\boldsymbol{i}}{2}(a+2l+c)v_{0,l}\in W, \bar{u}_{0,l,0}-\frac{\boldsymbol{i}}{2}(a+2l+c)w_{0,l,0}\in W$.

If $c\ne0,1$ or $a+c(1-\varepsilon)\notin2\bZ$, from
\begin{align}
&(D_{t^{k+1}\xi_1}+\boldsymbol{i}D_{t^{k+1}\xi_2})(u_{0,l}-\frac{\boldsymbol{i}}{2}(a+2l+c(1-\varepsilon))v_{0,l})=\boldsymbol{i}ckv_{0,l+k,\varepsilon}^+,\label{eq2}\\
&(D_{t^{k+1}\xi_1}+\boldsymbol{i}D_{t^{k+1}\xi_2})(\bar{u}_{0,l,\varepsilon}-\frac{\boldsymbol{i}}{2}(a+2l+c(1-\varepsilon))w_{0,l,\varepsilon})=\boldsymbol{i}(c-1)kw_{0,l+k+1-\varepsilon}^+,\label{eq3}\\
&D_{t\xi_3}(u_{0,l}-\frac{\boldsymbol{i}}{2}(a+2l+c(1-\varepsilon))v_{0,l})=\bar{u}_{0,l,\varepsilon}-\frac{\boldsymbol{i}}{2}(a+2l+c(1-\varepsilon))w_{0,l,\varepsilon},\label{eq4}\\
&D_{t\xi_3}(\bar{u}_{0,l,\varepsilon}-\frac{\boldsymbol{i}}{2}(a+2l+c(1-\varepsilon))w_{0,l,\varepsilon})\nonumber\\
=&-\frac{1}{2}(a+2l+c(1-\varepsilon))(u_{0,l+1-\varepsilon}-\frac{\boldsymbol{i}}{2}(a+2l+(c+2)(1-\varepsilon))v_{0,l+1-\varepsilon}),\label{eq5}
\end{align}
we get $W\cap(\bC v_{0,l,\varepsilon}^++\bC w_{0,l}^+)\ne0$, and hence $W=\cF_a(0,c,\varepsilon)$ by Lemma \ref{condition}.

Suppose $c=a=0$. From (\ref{eq3}) and (\ref{eq4}), we know that $W\cap(\bC v_{0,l,\varepsilon}^++\bC w_{0,l}^+)\ne0$, and hence $W=\cF_0(0,0,\varepsilon)$ or $W$ has codimension $1$ with $v_{0,0}\notin W$. For the later case $W$ is clear maximal. Denote

Now suppose $c=1, a=1-\varepsilon$. In this case, one can easily check that $\bC \bar{u}_{0,\varepsilon-1,\varepsilon}$, on which $\cK$ acts trivially, is a submodule with $\cF_{1-\varepsilon}(0,1,\varepsilon)/\bC\bar{u}_{0,\varepsilon-1,\varepsilon}\cong\cF'_0(0,0,\varepsilon)$.
\end{proof}

Now let us consider $\cF_a(m,c,\varepsilon)$ with $m>0$.

\begin{lemma}\label{submod}
Let $m>0$ and $W$ be any nonzero submodule of $\cF_a(m,c,\varepsilon)$. Then $W\cap(\bC v_{0,l,\varepsilon}^++\bC w_{0,l}^+)\ne0$ for some $l$.
\end{lemma}
\begin{proof}
Since $\cF_a(m,c,1)$ is a weight $\bC D_t+\bC D_{t\xi_1\xi_2}$ module,
\begin{align*}
W\cap&\big(\bC v_{r,l}+\bC v_{r+1,l,1}^++\bC v_{r-1,l,1}^-+\bC w_{r,0,1}+\bC w_{r+1,l}^++\bC w_{r-1,l}^-\\
&+\bC(u_{r,l}-\frac{\boldsymbol{i}}{2}(a+2l)v_{r,l})+\bC(\bar{u}_{r,l,1}-\frac{\boldsymbol{i}}{2}(a+2l)w_{r,l,1})\big)\ne0
\end{align*}
for some $r$ and some $l$. From the actions of $D_{t\xi_1}+\boldsymbol{i}D_{t\xi_2}$, we may assume $W\cap\big(\bC v_{r+1,l,1}^++\bC w_{r+1,l}^++\bC(u_{r,l}-\frac{\boldsymbol{i}}{2}(a+2l)v_{r,l})+\bC(\bar{u}_{r,l,1}-\frac{\boldsymbol{i}}{2}(a+2l)w_{r,l,1})\big)\ne0$. Applying $D_{t\xi_1\xi_3+\boldsymbol{i}D_{t\xi_2\xi_3}}$ several times, we have either $W\cap(\bC v_{0,l,1}^++\bC w_{0,l}^+)\ne0$ or $W\cap\big(\bC(u_{0,l}-\frac{\boldsymbol{i}}{2}(a+2l)v_{0,l}-\boldsymbol{i}w_{1,l}^+)+\bC(\bar{u}_{0,l,1}-\frac{\boldsymbol{i}}{2}(a+2l)w_{0,l,1}-\frac{\boldsymbol{i}}{2}(a+2l)v_{1,l,1}^+)\big)\ne0$. For the later case, from
\begin{align*}
&D_{t\xi_3}(u_{0,l}-\frac{\boldsymbol{i}}{2}(a+2l)v_{0,l}-\boldsymbol{i}w_{1,l}^+)=\bar{u}_{0,l,1}-\frac{\boldsymbol{i}}{2}(a+2l)w_{0,l,1}-\frac{\boldsymbol{i}}{2}(a+2l)v_{1,l,1}^+,\\
&D_{t\xi_1\xi_2\xi_3}(\bar{u}_{0,l,1}-\frac{\boldsymbol{i}}{2}(a+2l)w_{0,l,1}-\frac{\boldsymbol{i}}{2}(a+2l)v_{1,l,1}^+)=\boldsymbol{i}m(u_{0,l}-\frac{\boldsymbol{i}}{2}(a+2l)v_{0,l}-\boldsymbol{i}w_{1,l}^+),\\
&(D_{t^{k+1}\xi_1}+\boldsymbol{i}D_{t^{k+1}\xi_2})(u_{0,l}-\frac{\boldsymbol{i}}{2}(a+2l)v_{0,l}-\boldsymbol{i}w_{1,l}^+)=\boldsymbol{i}(c-m)kv_{0,l+k,1}^+,\\
&(D_{t^{k+1}\xi_1}+\boldsymbol{i}D_{t^{k+1}\xi_2})(\bar{u}_{0,l,1}-\frac{\boldsymbol{i}}{2}(a+2l)w_{0,l,1}-\frac{\boldsymbol{i}}{2}(a+2l)v_{1,l,1}^+)=\boldsymbol{i}(c-m-1)kw_{0,l+k}^+,
\end{align*}
we get $W\cap(\bC v_{0,l,1}^++\bC w_{0,l}^+)\ne0$.

For $\cF_a(m,c,0)$, since it is a weight $\bC D_t+\bC D_{\xi_1\xi_2}$ module, we know that there exists some $l$ and some $r$, such that either $W\cap(\bC v_{r+1,l,0}^++\bC v_{r-1,l,0}^-+\bC w_{r,l,0}+\bC(\bar{u}_{r,l,0}-\frac{\boldsymbol{i}}{2}(a+2l+c)w_{r,l,0}))\ne0$ or $W\cap(\bC v_{r,l}+\bC w_{r+1,l}^++\bC w_{r-1,l}^-+\bC(u_{r,l}-\frac{\boldsymbol{i}}{2}(a+2l+c)v_{r,l}))\ne0$. Applying $D_{t\xi_1}+\boldsymbol{i}D_{t\xi_2}$, we get $W\cap(\bC v_{r+1,l,0}^++\bC(\bar{u}_{r,l,0}-\frac{\boldsymbol{i}}{2}(a+2l+c)w_{r,l,0}))\ne0$ or $W\cap(\bC w_{r+1,l}^++\bC(u_{r,l}-\frac{\boldsymbol{i}}{2}(a+2l+c)v_{r,l}))\ne0$ for some $r$ and $l$. From the action of $D_{\xi_1\xi_3}+\boldsymbol{i}D_{\xi_2\xi_3}$, we know that $v_{0,l,0}^+\in W$ or $w_{0,l}^+\in W$ or $\bar{u}_{0,l,0}-\frac{\boldsymbol{i}}{2}(a+2l+c)w_{0,l,0}-\frac{\boldsymbol{i}}{2}(a+2l+c)v_{1,l,0}^+\in W$ or $u_{0,l}-\frac{\boldsymbol{i}}{2}(a+2l+c)v_{0,l}-\boldsymbol{i}w_{1,l}^+\in W$. Then lemma follows from
\begin{align*}
&D_{t\xi_3}(u_{0,l}-\frac{\boldsymbol{i}}{2}(a+2l+c)v_{0,l}-\boldsymbol{i}w_{1,l}^+)=\bar{u}_{0,l,0}-\frac{\boldsymbol{i}}{2}(a+2l+c)w_{0,l,0}-\frac{\boldsymbol{i}}{2}(a+2l+c)v_{1,l,0}^+,\\
&D_{t\xi_1\xi_2\xi_3}(\bar{u}_{0,l,0}-\frac{\boldsymbol{i}}{2}(a+2l+c)w_{0,l,0}-\frac{\boldsymbol{i}}{2}(a+2l+c)v_{1,l,0}^+)\\
=&\boldsymbol{i}m(u_{0,l+2}-\frac{\boldsymbol{i}}{2}(a+2l+4+c)v_{0,l+2}-\boldsymbol{i}w_{1,l+2}^+),\\
&(D_{t^{k+1}\xi_1}+\boldsymbol{i}D_{t^{k+1}\xi_2})(u_{0,l}-\frac{\boldsymbol{i}}{2}(a+2l+c)v_{0,l}-\boldsymbol{i}w_{1,l}^+)=\boldsymbol{i}(c-m)kv_{0,l+k,0}^+,\\
&(D_{t^{k+1}\xi_1}+\boldsymbol{i}D_{t^{k+1}\xi_2})(\bar{u}_{0,l,0}-\frac{\boldsymbol{i}}{2}(a+2l+c)w_{0,l,0}-\frac{\boldsymbol{i}}{2}(a+2l+c)v_{1,l,0}^+)\\
=&\boldsymbol{i}(c-m-1)kw_{0,l+k+1}^+.\qedhere
\end{align*}
\end{proof}

When $c=-m<0$, one can easily check that there is a unique proper submodule of $\cF_a(m,-m,\varepsilon)$, which intersects nontrivially with $(\bC v_{0,l,\varepsilon}^++\bC w_{0,l}^+)$ for some $l$, denote by $\cF'_a(m,\varepsilon)$. More precisely, we have
\[\cF'_a(m,\varepsilon)
=\sum\limits_{l\in\bZ}\sum\limits_{r=0}^{2m}(\bC\mathbf{u}_{r,l,\varepsilon}+\bC\bar{\mathbf{u}}_{r,l,\varepsilon})+\sum\limits_{l\in\bZ}\sum\limits_{r=0}^{2m+2}(\bC\mathbf{v}_{r,l,\varepsilon}+\bC\mathbf{w}_{r,l,\varepsilon}),\]
where
\begin{align*}
\mathbf{u}_{r,l}&=(2m+1-r)u_{r,l,\varepsilon}+\frac{\boldsymbol{i}(2m+1-r)}{2}(a+2l+c(1-\varepsilon))v_{r,l}+\boldsymbol{i}rw_{r-1,l}^-,\\
\bar{\mathbf{u}}_{r,l}&=(2m+1-r)\bar{u}_{r,l,\varepsilon}+\frac{\boldsymbol{i}(2m+1-r)}{2}(a+2l+c(1-\varepsilon))w_{r,l,\varepsilon}+\frac{\boldsymbol{i}}{2}(a+2l+c(1-\varepsilon))rv_{r-1,l,\varepsilon}^-\\
\mathbf{v}_{r,l,\varepsilon}&=(2m-r+2)(2m-r+1)v_{r,l,\varepsilon}^+-2r(2m-r+2)w_{r-1,l,\varepsilon}-r(r-1)v_{r-2,l,\varepsilon}^-,\\
\mathbf{w}_{r,l,\varepsilon}&=(2m-r+2)(2m-r+1)w_{r,l}^+-r(2m-r+2)(a+2l+c(1-\varepsilon))v_{r-1,l}-r(r-1)w_{r-2,l}^-.
\end{align*}

\begin{theorem}
Let $m>0$.
\begin{enumerate}
\item If $c+m\ne0$, then $\cF_a(m,c,\varepsilon)$ is a simple $\cK(N;\epsilon)$ module.
\item When $c=-m$, $\cF'_a(m,\varepsilon)$ is simple and is the unique nonzero proper $\cK(N;\epsilon)$ submodule of $\cF_a(m,-m,\varepsilon)$.
\item Suppose $c=-m<-1$. Then $\cF_a(m,-m,\varepsilon)/\cF'_a(m,\varepsilon)\cong\Pi(\cF'_{a+1-\varepsilon}(m-1,\varepsilon))$.
\item $\cF_a(1,-1,\varepsilon)/\cF'_a(1,\varepsilon)\cong\Pi(\cF_{a+1-\varepsilon}(0,0,\varepsilon))$.
\item $\cF_a(\frac{1}{2},-\frac{1}{2},\varepsilon)/\cF'_a(\frac{1}{2},\varepsilon)$ is simple.
\end{enumerate}
\end{theorem}
\begin{proof}
Statement 1 follows from Lemma \ref{condition} and Lemma \ref{submod}. It is obvious that $\cF'_a(m,\varepsilon)$ is simple and Statement 2 follows from Lemma \ref{submod}. For Statement 3 and 4, it is easy to check that the map defined by mapping $\bar{v}_{0,l}$ to $v_{0,l,\varepsilon}^+$ provides the isomorphism. For Statement 5, $\cF_a(\frac{1}{2},-\frac{1}{2},\varepsilon)/\cF'_a(\frac{1}{2},\varepsilon)=\sum\limits_{l\in\bZ}(\bC\bar{v}_{0,l}+\bC\bar{v}_{1,l}+\bC\bar{w}_{0,l,\varepsilon}+\bC\bar{w}_{1,l,\varepsilon})$ with actions
\begin{align*}
&D_{t^{k+1}}\bar{v}_{r,l}=(a+2l+\frac{3}{2}k)\bar{v}_{r,l+k},\, (D_{t^{k+1}\xi_1}+\boldsymbol{i}D_{t^{k+1}\xi_2})\bar{v}_{r,l}=0,\\
&(D_{t^{k+1}\xi_1}-\boldsymbol{i}D_{t^{k+1}\xi_2})\bar{v}_{r,l}=2(r-1)\bar{w}_{1,l+k,\varepsilon}, \, D_{t^{k+1}\xi_3}\bar{v}_{r,l}=\bar{w}_{r,l+k,\varepsilon},\\
&D_{t^{k+1}\xi_1\xi_2}\bar{v}_{r,l}=\boldsymbol{i}(r-\frac{1}{2})\bar{v}_{r,l+k+1-\varepsilon}, \, (D_{t^{k+1}\xi_1\xi_3}+\boldsymbol{i}D_{t^{k+1}\xi_2\xi_3})\bar{v}_{r,l}=r\bar{v}_{0,l+k+1-\varepsilon},\\
&(D_{t^{k+1}\xi_1\xi_3}-\boldsymbol{i}D_{t^{k+1}\xi_2\xi_3})\bar{v}_{r,l}=(r-1)\bar{v}_{1,l+k+1-\varepsilon},\, D_{t^{k+1}\xi_1\xi_2\xi_3}\bar{v}_{r,l}=0,\\
&D_{t^{k+1}}\bar{w}_{r,l,\varepsilon}=(a+2l+\frac{1}{2}k+1-\varepsilon)\bar{w}_{r,l+k,\varepsilon},\\
&(D_{t^{k+1}\xi_1}+\boldsymbol{i}D_{t^{k+1}\xi_2})\bar{w}_{r,l,\varepsilon}=r(a+2l+k+\frac{3}{2}(1-\varepsilon))\bar{v}_{0,l+k+1-\varepsilon},\\
&(D_{t^{k+1}\xi_1}-\boldsymbol{i}D_{t^{k+1}\xi_2})\bar{w}_{r,l,\varepsilon}=(r-1)(a+2l+k+\frac{3}{2}(1-\varepsilon))\bar{v}_{1,l+k+1-\varepsilon},\\
&D_{t^{k+1}\xi_3}\bar{w}_{r,l,\varepsilon}=-\frac{1}{2}(a+2l+k+\frac{3}{2}(1-\varepsilon))\bar{v}_{r,l+k+1-\varepsilon},\\
&D_{t^{k+1}\xi_1\xi_2}\bar{w}_{r,l,\varepsilon}=\boldsymbol{i}(r-\frac{1}{2})\bar{w}_{r,l+k+1-\varepsilon,\varepsilon},\, D_{t^{k+1}\xi_1\xi_2\xi_3}\bar{w}_{r,l,\varepsilon}=\boldsymbol{i}(\frac{1}{2}-r)\bar{v}_{r,l+k+2(1-\varepsilon)},\\ &(D_{t^{k+1}\xi_1\xi_3}+\boldsymbol{i}D_{t^{k+1}\xi_2\xi_3})\bar{w}_{r,l,\varepsilon}=-r\bar{w}_{0,l+k+1-\varepsilon,\varepsilon},\\
&(D_{t^{k+1}\xi_1\xi_3}-\boldsymbol{i}D_{t^{k+1}\xi_2\xi_3})\bar{w}_{r,l,\varepsilon}=(1-r)\bar{w}_{0,l+k+1-\varepsilon,\varepsilon},
\end{align*}
where $r=0,1$ and $l,k\in\bZ$. Then one can easily check that it is simple.
\end{proof}

\noindent{\bf Acknowledgement:}  R. L\"u is partially supported by NSF of China (Grant 12271383, 11971440). We thank Prof. Zelmanov for sharing their manuscript.

%\bibliography{mybibfile}

\section*{References}



\begin{thebibliography}{00}

\bibitem{A} M. Ademollo, et al. {\it Supersymmetric strings and color confinement.} Phys. Lett. {\bf 62B}, 105-110 (1976).

\bibitem{B} Y. Billig, {\it Jet modules.} Canad. J. Math. 59 (2007), no. 4, 721-729.

\bibitem{BF} Y. Billig, V. Futorny, {\it Classification of irreducible representations of Lie algebra of vector fields on a torus}, J. Reine Angew. Math.,  720(2016): 199-216.

%\bibitem{BF1} Y. Billig, V. Futorny. {\it Classification of simple cuspidal modules for solenoidal Lie algebras}. Israel J. Math. 222 (2017), no. 1, 109-123.
\bibitem{BFIK} Y. Billig, V. Futorny, K. Iohara, I. Kashuba. {\it Classification of simple strong Harish-Chandra modules}, arxiv:2006.05618.


\bibitem{CL} Y. Cai, R. L\"u. {\it Classification of simple Harish-Chandra modules over the Neveu-Schwarz algebra and its contact subalgebra}, J. Pure Appl. Algebra, 226 (2022), no. 3, Paper No. 106866.



\bibitem{CLL} Y. Cai, D. Liu, R. L\"{u}, {\it Classification of simple Harish-Chandra modules over the $N=1$ Ramond algebra}, J. Algebra, 567 (2021), 114-127.

\bibitem{CLW} Y. Cai, R. L\"{u}, Y. Wang, {\it Weight modules for map (super)algebra related to the Virasoro algebra}, J. Algebra, 570 (2021), 397-415.

\bibitem{CLX} Y. Cai, R L\"u, Y. Xue. {\it Classification of simle strong Harish-Chandra modules over the Lie superalgebra of vector fields on $\bC^{m|n}$}, arXiv: 2106.04801.

\bibitem{CP} V. Chari, A. Pressley, {\it Unitary representations of the Virasoro algebra and a
conjecture of Kac}, Compositio Mathematica, 67(1988), 315-342


\bibitem{DLM} P. Desrosiers, L. Lapointe,  P. Mathieu, {\it Superconformal field theory and Jack superpolynomials}, JHEP, { 09} (2012), 037.


\bibitem{Rao} S. E. Rao, {\it Partial classification of modules for Lie algebra of diffeomorphisms of $d$-dimensional torus}, J. Math. Phys. 45 (2004), no. 8, 3322-3333.

%\bibitem{FGM} V. Futorny, D. Grantcharov, V. Mazorchuk, {\it Weight modules over infinite dimensional Weyl algebras}, Proc. Amer. Math. Soc., 142 (2014), no. 9, 3049-3057.

\bibitem{GLS} P. Grozman, D. Leites, I. Schepochkina, {\it Lie superalgebras of string theories}, Acta Math. Vietnam, 26 (2001), no. 1, 27-63.

\bibitem{Hu} J. Humphreys, {\it Introduction to Lie algebras and representation theory}. GTM 9, Springer-Verlag, New York-Berlin, 1978. xii+171 pp. ISBN: 0-387-900053-5.

\bibitem{IK1}  K. Iohara, Y. Koga, {\it Representation theory of Neveu-Schwarz and Remond algebras I: Verma
modules}. Adv. Math. 177(2003), 61-69.

\bibitem{IK2} K. Iohara, Y. Koga,   {\it Representation theory of Neveu-Schwarz and Remond algebras II: Fock
modules}. Ann. Inst. Fourier 53 (2003), 1755-1818.

\bibitem{IK3} K. Iohara, Y. Koga, {\it Representaion theory of $N=2$ super Virasoro algebra: twisted sector}, J. Funct. Ana. 214 (2004), no. 2, 450-518.


\bibitem {Ka} V. Kac, {\it Some problems of infinite-dimensional Lie algebras and their representations},
Lecture Notes in Mathematics, 933 (1982), 117-126. Berlin, Heidelberg,
New York: Springer.


\bibitem {Ka1} V. Kac, {\it Superconformal algebras and transitive group actions on
quadrics}, Commun. Math. Phys. 186, (1997) 233-252.

\bibitem {KL} V. Kac,  J. van de Leuer, {\it On classification of superconformal
algebras}, Strings 88, Sinapore: World Scientific, 1988.




\bibitem{KS} I. Kaplansky, L. J. Santharoubane, {\it Harish-Chandra modules over the Virasoro algebra}, Infinite-dimensional groups with applications (Berkeley, Calif. 1984), 217--231, Math. Sci. Res. Inst. Publ., 4, Springer, New York, 1985.

%\bibitem{LPX} D. Liu, Y. Pei, L. Xia, Classification of simple weight modules for the N=2 superconformal algebra, arXiv:1904.08578.

%\bibitem{LG} G. Liu, X. Guo, {\it Harish-Chandra modules over generalized Heisenberg-Virasoro algebras}. Israel J. Math. 204 (2014), no. 1, 446-468.

%\bibitem{LZ} G. Liu, K. Zhao, {\it Irreducible Harish Chandra modules over the derivation algebras of rational quantum tori}. Glasg. Math. J. 55 (2013), no. 3, 677-693.

\bibitem{LuZ}  R. L\"u, K. Zhao, {\it Classification of irreducible weight modules over the twisted Heisenberg-Virasoro algebra}. Commun. Contemp. Math. 12 (2010), 183-205.

\bibitem{LPX} D. Liu, Y. Pei, L. Xia, {\it Classification of simple modules with finite-dimensional weight spaces for the $N = 2$ Ramond algebra}, arXiv: 190408578.

%\bibitem{LPX1} D. Liu, Y. Pei,  L. Xia, {\it Whittaker modules for the super-Virasoro algebras}, J. Algebra Appl.  18 (2019),  1950211.

%\bibitem{LPX2} D. Liu, Y. Pei, L. Xia, {\it Simple restricted modules for Neveu-Schwarz algebra}, J. Algebra, 546 (2020), 341-356.


\bibitem{Ma} O. Mathieu,  {\it Classification of Harish-Chandra
modules over the Virasoro Lie algebra}, Invent. Math. 107(1992), 225-234.

%\bibitem{MP} C. Martin, A. Piard, {\it Nonbounded
%indecomposable admissible modules over the Virasoro algebra}, Lett. Math. Phys. 23(1991), 319-324.
\bibitem{M} T. Moons. {\it On the weight spaces of Lie superalgebra modules}, J. Algebra, 147(2) (1992), 283-323.

\bibitem{MZe} C. Mart\'{\i}nez, E. Zelmanov, {\it Graded modules over superconformal algebras}, Non-Associative and Non-Commutative Algebra and Operator Theory. Springer International Publishing, 2016, 41-53
\bibitem{MZe1} C. Mart\"{\i}nez, E. Zelmanov, {\it Brackets, superalgebras and spectral gap}. S\~{a}o Paulo J. Math. Sci. 13 (2019), no. 1, 112-132.

\bibitem{MZ} V. Mazorchuk, K. Zhao, {\it Supports of weight modules over Witt algebras. Proc. Roy. Soc. Edinburgh Sect. A}, 141 (2011), no. 1, 155-170.

\bibitem{NS} A. Neveu, J. H. Schwarz, {\it Factorizable dual model of pions.} Nucl. Phys. B {\bf 31} (1), 86-112 (1971).

\bibitem{N} J. Nilsson,  {\it $U(\mathfrak{h})$-free modules and coherent families}, J. Pure Appl. Algebra 220 (2016), no. 4, 1475-1488.

\bibitem{R} P. Ramond, {\it Dual theory for free fermions.} Phys. Rev. D (3) {\bf 3}, 2415-2418 (1971).

\bibitem{Sc} K. Schoutens, {\it A nonlinear representation of the $d=2$ $SO(4)$ extended superconformal algebra.} Phys. Lett. B
    {\bf 194}, 75-80 (1987).

%\bibitem{SS} A. Schwimmer, N. Seiberg, {\ite Comments on the N = 2, 3, 4 superconformal algebras in two dimensions}, Physics Letters B, 184(2-3):191¨C196, 1987.


\bibitem{ST1} A. Semikhatov, I. Tipunin, {\it The structure of Verma modules over the $N=2$ superconformal algebra}, Comm. Math. Phys. 195 (1998), no. 1, 129-173.

\bibitem{ST2} A. Semikhatov, I. Tipunin, {\it Resolutions and characters of irreducible representations of the $N=2$ superconformal algebra}, Nuclear Phys. B 536 (1999), no. 3, 617-656.

\bibitem{S2} Y. Su, {\it Classification of Harish-Chandra modules over the super-Virasoro algebras},  Commun. Alg.  23(10) (1995), 3653-3675.


\bibitem{S1}  Y. Su, {\it A classification of indecomposable $sl_2({\mathbb C})$-modules and a conjecture of Kac on irreducible modules over the Virasoro algebra}, J. Alg, 161(1993), 33-46.

\bibitem{XL1} Y. Xue, R. L\"{u}, {\it Simple weight modules with finite dimensional weight spaces over Witt superalgebras},  J. Algebra 574 (2021), 92-116.

%\bibitem{XL} Y. Xue, R. L\"{u}, {\it Classification of simple bounded weight modules of the Lie algebra of vector fields on $\bC^n$}, Israel J. Math. 253 (2023), no. 1, 445-468.

\end{thebibliography}

\end{document}